%% file: main.tex
\documentclass[11pt, a4paper]{amsart}

\usepackage[
  backend=biber,
  style=alphabetic,
  sorting=nty,
  maxcitenames=50,
  maxnames=50
]{biblatex}

\renewbibmacro{in:}{}
\AtEveryBibitem{%
  \clearfield{issn} 
  \clearfield{doi} 
  \clearfield{isbn}
  \ifentrytype{online}{}{
    \clearfield{url}
  }
}

\input{preamble.tex}

\title{A note on cubic fourfolds containing several planes}
\author{Moritz Hartlieb}
\address{Mathematisches Institut, Universitat Bonn, Endenicher Allee 60, 53115
Bonn, Germany}
\email{hartlieb@math.uni-bonn.de}

\begin{document}

\begin{abstract}
  We study the geometry, Hodge theory and derived category of cubic fourfolds
  containing several planes and their associated twisted K3 surfaces. We focus on the
  case of two planes intersecting along a line.
\end{abstract}

\maketitle

\input{intro}
\input{basics}

\input{eckardt}

\input{geometry}
\input{transcendental}

\input{derived}
\input{others}
\appendix

\input{latticecomp}

\printbibliography

\end{document}

%% file: preamble.tex
\usepackage{amsfonts}
\usepackage{amssymb}
\usepackage{amsthm}
\usepackage{mathtools}
\usepackage{comment}
\usepackage{tikz-cd}
\usepackage[all]{xy}

\usepackage{enumitem}
\setenumerate{label=\rm{(\roman*)}}

\usepackage{todonotes}

\usepackage[T2A]{fontenc}
\DeclareSymbolFont{cyrillic}{T2A}{cmr}{m}{n}
\DeclareMathSymbol{\Sha}{\mathalpha}{cyrillic}{216}

\PassOptionsToPackage{pdfusetitle,colorlinks}{hyperref}
\usepackage{bookmark}
\hypersetup{
  linkcolor={red!70!black},
  citecolor={green!70!black},
  urlcolor={blue!80!black}
}

\newtheorem{theorem}{Theorem}[section]

\newtheorem{lemma}[theorem]{Lemma}
\newtheorem{proposition}[theorem]{Proposition}
\newtheorem{question}[theorem]{Question}
\newtheorem{corollary}[theorem]{Corollary}

\theoremstyle{definition}

\newtheorem{example}[theorem]{Example}
\newtheorem{remark}[theorem]{Remark}

\DeclareMathOperator{\End}{End}
\DeclareMathOperator{\Pic}{Pic}

\DeclareMathOperator{\Aut}{Aut}

\DeclareMathOperator{\Br}{Br}
\DeclareMathOperator{\NS}{NS}

\DeclareMathOperator{\Coh}{Coh}

\DeclareMathOperator{\Ram}{Ram}
\DeclareMathOperator{\disc}{disc}
\DeclareMathOperator{\Sing}{Sing}
\DeclareMathOperator{\Hom}{Hom}
\DeclareMathOperator{\Bl}{Bl}

\renewcommand{\to}{\xymatrix@1@=15pt{\ar[r]&}}
\renewcommand{\rightarrow}{\xymatrix@1@=15pt{\ar[r]&}}
\renewcommand{\leftarrow}{\xymatrix@1@=15pt{&\ar[l]}}
\renewcommand{\mapsto}{\xymatrix@1@=15pt{\ar@{|->}[r]&}}
\renewcommand{\twoheadrightarrow}{\xymatrix@1@=18pt{\ar@{->>}[r]&}}
\renewcommand{\hookrightarrow}{\;\xymatrix@1@=15pt{\ar@{^(->}[r]&}}
\newcommand{\hook}{\xymatrix@1@=15pt{\ar@{^(->}[r]&}}
\newcommand{\congpf}{\xymatrix@L=0.6ex@1@=15pt{\ar[r]^-\sim&}}
\renewcommand{\cong}{\simeq}
\renewcommand{\tilde}{\widetilde}

\makeatletter
\newcommand*{\da@rightarrow}{\mathchar"0\hexnumber@\symAMSa 4B }
\newcommand*{\da@leftarrow}{\mathchar"0\hexnumber@\symAMSa 4C }
\newcommand*{\xdashrightarrow}[2][]{%
  \mathrel{%
    \mathpalette{\da@xarrow{#1}{#2}{}\da@rightarrow{\,}{}}{}%
  }%
}
\newcommand{\xdashleftarrow}[2][]{%
  \mathrel{%
    \mathpalette{\da@xarrow{#1}{#2}\da@leftarrow{}{}{\,}}{}%
  }%
}
\newcommand*{\da@xarrow}[7]{%
  \sbox0{$\ifx#7\scriptstyle\scriptscriptstyle\else\scriptstyle\fi#5#1#6\m@th$}%
  \sbox2{$\ifx#7\scriptstyle\scriptscriptstyle\else\scriptstyle\fi#5#2#6\m@th$}%
  \sbox4{$#7\dabar@\m@th$}%
  \dimen@=\wd0 %
  \ifdim\wd2 >\dimen@
    \dimen@=\wd2 %
  \fi
  \count@=2 %
  \def\da@bars{\dabar@\dabar@}%
  \@whiledim\count@\wd4<\dimen@\do{%
    \advance\count@\@ne
    \expandafter\def\expandafter\da@bars\expandafter{%
      \da@bars
      \dabar@
    }%
  }%
  \mathrel{#3}%
  \mathrel{%
    \mathop{\da@bars}\limits
    \ifx\\#1\\%
    \else
      _{\copy0}%
    \fi
    \ifx\\#2\\%
    \else
      ^{\copy2}%
    \fi
  }%
  \mathrel{#4}%
}
\makeatother

%% file: intro.tex
\section{Introduction}

The aim of this note is to study the geometry of smooth, complex cubic fourfolds containing two
(and more) planes. There are three cases to consider:
\begin{enumerate}
  \item The two planes are disjoint;
  \item The two planes meet in a point;
  \item The two planes meet along a line.
\end{enumerate}
These $18$-dimensional families form three irreducible divisors in the Hassett divisor $C_8$ of cubic
fourfolds containing a plane. It is conjectured that a very general member of the family of cubic fourfolds containing a plane is
irrational, see \cite{hassettratio} for a survey on the rationality question for
cubic fourfolds. However, if a cubic fourfold contains two disjoint planes, it is
easily seen to be rational. As the families of cubic fourfolds containing two
planes meeting in a point or a line both contain the family of Eckardt cubic
fourfolds, the very general member of each of these families is conjectured to
be irrational, see \cite{maxalglaza} for a detailed study of Eckardt cubic
fourfolds with respect to rationality questions.

To a smooth cubic fourfold $X$ containing a plane $P \subset X$, one associates
a certain twisted K3 surface $S_P$ with a Brauer class $\alpha_P \in \Br(S_P)[2]$.
By a theorem of Kuznetsov \cite[Thm.\ 4.3]{kuz4fold}, refined by Moschetti \cite[Thm.\ 1.2]{moschetti}, there is
an equivalence
$$\mathcal A_X \cong D^b(S_P, \alpha_P),$$
where $\mathcal A_X = \langle \mathcal O_X,  \mathcal O_X(1), \mathcal O_X(2) \rangle^\perp \subset D^b(X)$
is the Kuznetsov component of $X$.
Then, if the smooth cubic fourfold $X$ contains two planes $P_1, P_2 \subset X$,
one obtains two twisted K3 surfaces $(S_{P_1}, \alpha_{P_1})$ and $(S_{P_2},
  \alpha_{P_2})$. The aim of this note is to study their geometric and Hodge-theoretic relation.
As an immediate consequence of Kuznetsov's theorem, we obtain:
\begin{theorem}[{\cite{kuz4fold, moschetti}}]
  Let $X$ be a smooth cubic fourfold containing two planes $P_1, P_2 \subset
    X$, then there exists a
  twisted derived equivalence
  \begin{equation}D^b(S_{P_1}, \alpha_{P_1}) \cong D^b(S_{P_2},
    \alpha_{P_2}).\label{eq:twisted}
  \end{equation}
\end{theorem}
If the two planes are disjoint, it is known that the two associated K3
surfaces are isomorphic, see \cite[§3, App.]{voisintorelli}.
In the other two cases, we show that the associated K3 surfaces are not
isomorphic.
\begin{theorem}[Thm.\ \ref{thm:notiso} and Thm.\ \ref{thm:notiso2}]   \label{thm:notisomain}
  Let $X$ be a cubic fourfold containing two non-disjoint planes $P_1, P_2 \subset X$, and assume that $X$ is very general with this property.  Then the associated K3 surfaces $S_{P_1}$ and $S_{P_2}$ are not isomorphic.
\end{theorem}
This yields a negative answer to a question posed in \cite[Quest.\ 5.1]{kkmquestion}\footnote{Thanks to Ziqi Liu for the reference.}.
In fact, the associated K3 surfaces admit no non-trivial Fourier--Mukai partners and are thus not even derived equivalent.

Next, consider the special case that the planes $P_1$ and $P_2$ intersect along a line $L =
  P_1 \cap P_2 \subset X$. The main results of this note are summarized by the
following four theorems, explaining the relation between the associated twisted
K3 surfaces $(S_{P_1}, \alpha_{P_1})$ and $(S_{P_2}, \alpha_{P_2})$ via a geometric correspondence.
First, we note that they are Tate--\v{S}afarevi\v{c} twists of the same elliptic K3 surface with a
section.
\begin{theorem}[{Cor.\ \ref{cor:commonjac} and Thm.\ \ref{thm:descbrauer}}]
  \label{thm:reltodp}
  Let $X$ be a cubic fourfold containing two planes $P_1, P_2 \subset
    X$ intersecting along a line, and assume that $X$ is very general with this property. Then the K3 surfaces $S_{P_1}$ and $S_{P_2}$ admit (unique) elliptic
  fibrations with isomorphic Jacobian K3 surfaces $$S \coloneqq J(S_{P_1} / \mathbb
    P^1) \cong J(S_{P_2} / \mathbb P^1).$$ The corresponding Brauer classes
  $$\beta_i \coloneqq [S_{P_i}] \in \Sha(S) \cong \Br(S)$$
  satisfy
  $$\overline{\beta_2} = \alpha_{P_1} \in \Br(S_{P_1}) \text{ and }
    \overline{\beta_1} = \alpha_{P_2} \in \Br(S_{P_2}).$$
\end{theorem}
Here, we denote by $\overline{\beta_i} \in \Br(S_{P_j})$ the image of $\beta_i \in \Br(S)$ under the natural map
$$\Br(S) \simeq \Hom(T(S), \mathbb Q / \mathbb Z) \to \Hom(T(S_{P_j}), \mathbb Q / \mathbb Z) \simeq \Br(S_{P_j})$$
induced by the inclusion $T(S_{P_j}) \subset T(S)$, cf. \cite[Thm.\ 5.4.3]{caldararuthesis}.

In particular, one may view the existence of the twisted derived equivalence $(\ref{eq:twisted})$
as an instance of a theorem of Donagi and Pantev \cite[Thm.\ A]{donagipantev}. It would be interesting to compare the equivalence via the Kuznetsov component with the equivalence provided by Donagi--Pantev. As a first step in this direction, we compute the action of the former on the Mukai lattice of the twisted K3 surfaces in Appendix \ref{sec:latticecomp}.

Then, we construct a correspondence between the two K3 surfaces, which geometrically explains the existence of a rational Hodge isometry between the transcendental lattices.
\begin{theorem}[Thm.\ \ref{prop:descfl}]
  \label{thm:exofcorr}
  Let $X$ be a cubic fourfold containing two planes $P_1, P_2 \subset
    X$ intersecting along a line $L = P_1 \cap P_2$, and assume that $X$ is very general with this property. Then, there is a smooth minimal
  surface $F_L$ of Kodaira dimension one admitting
  two quotients $f_i \colon F_L \to S_{P_i}$ of degree two, which split\footnote{Recall that a morphism $f \colon X\to Y$ \emph{splits} a Brauer class $\alpha \in \Br(Y)$ if the pullback $f^* \alpha$ is trivial in $\Br(X)$.
  } the Brauer classes
  $\alpha_{P_i} \in \Br(S_{P_i})$, and an elliptic fibration $F_L
    \to L$ fitting into the commutative diagram
  $$
    \begin{tikzcd}
      S_{P_1} \arrow[rdd, "\text{ell.\ fib.}"'] & F_L \arrow[d, "\text{ell.\ fib.}"
        description] \arrow[r, "2:1"] \arrow[l, "2:1"'] & S_{P_2} \arrow[ldd,
        "\text{ell.\ fib.}"] \\
      & L \arrow[d, "2:1" description]                                                   &                                         \\
      &\, \mathbb P^1   ,  &
    \end{tikzcd}
  $$
  where the map $L \to \mathbb P^1$ is the Gauss map of $X$ restricted to $L$.

\end{theorem}

In other words, there are two involutions $\iota_1, \iota_2 \in \Aut(F_L)$ with $S_{P_i} = F_L / \iota_i$ that lift the covering involution of the Gauss map $L \to \mathbb P^1$ and thus identify pairs of fibers of the elliptic fibration $F_L \to L$ in two different ways. Note that the fixed point loci of $\iota_1$ and $\iota_2$, which are the ramification loci $\mathbb E \subset F_L$ of the double covers $F_L \to S_{P_i}$, coincide, as they are the fibers over the ramification locus of the Gauss map $L \to \mathbb P^1$ and thus coincide for $\iota_1$ and $\iota_2$.

By \cite[§1, Prop.\ 1]{voisintorelli} the Fano correspondence, i.e., the universal family of lines on $X$, induces, up to a global sign, a Hodge isometric embedding
$T(X) \subset T(S_{P_i})$ of index two, where $T(X) \coloneqq H^4_{\mathrm{alg}}(X)^\perp \subset H^4(X, \mathbb Z)(-1)$ denotes the transcendental part of the cohomology of $X$.
Using the correspondence $F_L$, one can interpret this embedding as an intersection of the pullbacks of the transcendental lattices of $S_{P_1}$ and $S_{P_2}$ to $F_L$:

\begin{theorem}[{Thm.\ \ref{thm:intersectionistx}}]
  Let $X$ be a cubic fourfold containing two planes $P_1, P_2 \subset X$ intersecting along a line $L = P_1 \cap P_2$, and assume that $X$ is very general with this property.
  The Fano correspondence induces an integral Hodge isometry
  $$T(S_{P_i}, \alpha_{P_i}) \cong (T(X), -(\,.\,)) \cong (f^*_1 T(S_{P_1}) \cap f_2^* T(S_{P_2}), 1/2(\,.\,)).$$
\end{theorem}

Furthermore, also the twisted derived equivalence (\ref{eq:twisted}) may be understood via the above
correspondence.
\begin{theorem}[{Thm.\ \ref{thm:derivedmain}}]
  \label{thm:derived}
  Let $X$ be a cubic fourfold containing two planes $P_1, P_2 \subset
    X$ intersecting along a line $L = P_1 \cap P_2$, and assume that $X$ is very general with this property. There are autoequivalences
  $\Phi_1, \Phi_2 \in \Aut(D^b(F_L))$ of order two and semiorthogonal
  decompositions of the respective equivariant categories
  $$D^b(F_L)^{\langle \Phi_i \rangle} = \langle D^b(S_{P_i},\alpha_{P_i}), D^b(\mathbb E) \rangle,$$
  where $\mathbb E \subset F_L$ is the ramification locus of either of the double
  covers $F_L \to S_{P_i}$. Moreover, there is an equivalence
  $$D^b(F_L)^{\langle \Phi_1 \rangle} \congpf D^b(F_L)^{\langle \Phi_2 \rangle},$$
  which respects the semiorthogonal decompositions.
\end{theorem}

Let us end this introduction by giving an overview of the structure
of this note.
We begin by recalling the construction of the twisted K3 surfaces associated to
cubic fourfolds containing a plane in Section \ref{sec:basics}. In order to prove Theorem
\ref{thm:notisomain}, we degenerate to Eckardt cubic fourfolds, whose properties
we recall in Section \ref{sec:eckardt}. Then, we focus on the case of two planes intersecting along a line. In Section \ref{sec:geometry}, we prove Theorem
\ref{thm:notisomain} in this case, construct the surface $F_L$ and prove the
geometric part of Theorem \ref{thm:exofcorr}. The relation between the transcendental lattices of $S_{P_i}$ and $F_L$ is
studied in Section \ref{sec:transcendental}. In Section \ref{sec:derived}, we
study derived categories associated to the surface $F_L$ and establish Theorem
\ref{thm:derived}. The other two cases, i.e., trivial and pointwise intersection, are discussed
briefly in Section \ref{sec:others}.  Finally, we conclude this note by
explicitly computing the action of the twisted derived equivalence between the
associated K3 surfaces in Appendix \ref{sec:latticecomp}.

We work over the field $\mathbb C$ of complex numbers. Unless stated otherwise, all
cubic fourfolds are assumed to be smooth.

\subsection*{Acknowledgements}
I would like to thank my advisor Daniel Huybrechts for his interest in this project as well as his comments on a preliminary version of this article, and Claire Voisin for the invitation to Paris and the opportunity to present this material at the s\'eminaire de g\'eom\'etrie alg\'ebrique. Moreover, I would like to thank Reinder Meinsma for interesting discussions on the content of this note, as well as the two anonymous referees for helpful comments.
This research was funded by the ERC Synergy Grant HyperK, Grant agreement ID 854361. I am grateful for the support provided by the International Max Planck Research School on Moduli Spaces at the Max Planck Institute for Mathematics in Bonn.

%% file: basics.tex
\section{Cubic fourfolds containing a plane}
\label{sec:basics}
In this section, we recall the construction of a twisted K3 surface associated
to a plane in a cubic fourfold. See \cite[Ch.\ 6.1]{huybook} for more details and references.

Let $P \subset X$ be a smooth cubic fourfold containing a plane. Projecting from
$P$ induces a quadric surface fibration
$$f_P \colon \Bl_P X \to \mathbb P^2,$$
which has singular fibers precisely along a sextic $C'_P \subset \mathbb P^2$.
Let $g_P \colon F'_P \to \mathbb P^2$ denote the relative Fano variety of
lines on fibers of $f_P$. The fibers of $f_P$ and $g_P$ over $x \in
  \mathbb{P}^2$ are as follows:

\begin{center}
  \begin{tabular}{c | c | c}

    $x \in \mathbb P^2$                & fiber of $f_P$                  & fiber of $g_P$                   \\
    \hline
    $x \not \in C'_P$                  & smooth quadric surface          & $\mathbb P^1 \sqcup \mathbb P^1$
    \\
    $x \in C'_P \setminus \Sing(C'_P)$ & cone over smooth conic          & $\mathbb P^1$                    \\
    $x \in \Sing(C'_P)$                & two planes meeting along a line & two planes
    meeting in a point
  \end{tabular}
\end{center}
\vspace{0.5em}

The smoothness of $X$ implies that the singularities of $C'_P$ are ordinary double points.
For later use, we highlight the following characterization of the singularities of $C'_P$, which directly follows from the above description.
\begin{lemma}
  \label{lem:singularitiesofcurve} The singularities of $C'_P$ correspond to pairs of planes $P', P'' \subset X$ for which there is a linear three-space $\Pi \subset \mathbb P^5$ with $X \cap \Pi = P \cup P' \cup P''$.
\end{lemma}

Let $S'_P$ denote the Stein factorization of $g_P \colon F_P' \to \mathbb P^2$.
Note that $S'_P$ is a double cover of $\mathbb P^2$, branched along the
sextic curve $C'_P \subset \mathbb P^2$. Hence, $S'_P$ is a singular K3 surface with ordinary double points over the ordinary double points of $C'_P$.
The resolution $S_P$ of $S'_P$ is a K3 surface admitting a double cover
$$S_P \to   \Bl_{\Sing_{C'_P}} \mathbb P^2,$$
branched along the strict transform $C_P \subset \Bl_{\Sing_{C'_P}} \mathbb P^2$ of $C'_P$.

Away from the singular locus of $C'_P \subset S'_P$, the relative Fano variety
$F'_P \to S'_P$ is a Brauer--Severi scheme, i.e., \'etale locally isomorphic to a projective bundle.
In \cite[Prop.\ 4.7]{moschetti}, Moschetti has shown that this extends to a
Brauer--Severi scheme $F_P \to S_P$, see also \cite[Prop.\ 4.4]{kuzsing}.

For later use, we recall the description of $F_P$ over the singular points of
$C'_P$.
\begin{lemma}[{\cite{moschetti}, \cite[Prop.\ 4.4]{kuzsing}}]
  \label{lem:desc_resolution}
  Let $\nu \colon S_P \to S'_P$ denote the map resolving the singularities of
  $S'_P$.
  The restriction of the Brauer--Severi scheme $F_P \to S_P$ to the
  exceptional curves $\mathbb P^1 \cong \nu^{-1}(x)$ for $x \in \Sing(S'_P)$ is
  isomorphic to the natural projection
  $\Bl_{\mathrm{point}} \mathbb P^2 \to \mathbb P^1$.
\end{lemma}

Let $\alpha_P = [F_P] \in \Br(S_P)[2]$ denote the Brauer class corresponding to the
Brauer--Severi scheme $F_P \to S_P$. Let $H^4(X, \mathbb Z)(-1)$ denote the Tate twist of $H^4(X, \mathbb Z)$. Adapting the arguments from the non-singular case, one can show that the transcendental part of $H^4(X, \mathbb Z)(-1)$ agrees with the transcendental lattice of the twisted K3 surface $(S_P, \alpha_P)$:
\begin{proposition}[{\cite[§1, Prop.\ 1]{voisintorelli}, cf.\ \cite[Prop.\ 6.1.18]{huybook}}]
  \label{prop:fanohodge}
  The Fano correspondence realizes the transcendental lattice $$T(X) \coloneqq
    H^4_{\mathrm{alg}}(X, \mathbb Z)^{\perp} \subset
    H^4(X, \mathbb Z)(-1)$$ as the kernel of $\alpha_P \in \Br(S_P) \cong
    \Hom(T(S_P), \mathbb Q/ \mathbb Z),$ i.e., we have a Hodge isometric
  embedding
  $$(T(X),-(\,.\,) ) \cong T(S_P, \alpha_P) \hookrightarrow T(S_P) \overset{\alpha_P}\to \mathbb Z / 2 \mathbb Z.$$
\end{proposition}
\begin{proof}
  When the discriminant sextic $C_P'$ is smooth, this is follows from \cite[Prop.\ 6.1.18]{huybook}. The general case follows by a specialization argument.
  We sketch the proof in the case that $C_P$ has one ordinary double point, the general case is left to the reader. By Lemma~\ref{lem:singularitiesofcurve}, there then is  unique linear three space $\Pi \subset \mathbb P^5$ and two planes $P', P''$ such that $$X \cap \Pi = P \cup P' \cup P''.$$
  By realizing $X$ as a hyperplane section of a sufficiently general cubic fivefold $Y \subset \mathbb P^6$ and the considering the family of hyperplane sections of $Y$ containing the plane $P \subset X \subset Y$, one obtains a family
  $$P \times \Delta \subset \mathcal {X} \to \Delta$$
  of smooth cubic fourfolds containing a plane over the disc such that $\mathcal{X}_0 \simeq X$ and $\mathcal {X}_t$ contains exactly one plane for $t \neq 0$. Moreover, we get a family of Fano varieties and singular K3 surfaces associated to the cubic fourfolds $\mathcal X_t$ for $t \in \Delta$:
  \begin{equation}\label{eq:singfam}\begin{tikzcd}
      \mathcal F_P' \arrow[r, hook] \arrow[d] & \mathcal F = F(\mathcal X / \Delta) \arrow[d] \\
      \mathcal S_P' \arrow[r]                 & \Delta.
    \end{tikzcd}\end{equation}
  Note that $\mathcal S'_{P, t}$ is a smooth K3 surface for $t \neq 0$.
  As explained in \cite[Sec.\ 4]{moschetti}, see also \cite[Sec.\ 4]{kuzsing}, one can then apply a birational modification along the dual plane $(P')^\vee \subset F(X) \simeq \mathcal F_0$ to obtain families
  $$\begin{tikzcd}
      \mathcal F_P \arrow[r, hook] \arrow[d] & \widetilde{\mathcal F} \arrow[d] \\
      \mathcal S_P \arrow[r]                 & \Delta,
    \end{tikzcd}$$
  isomorphic to $(\ref{eq:singfam})$ over $\Delta \setminus \{0\}$ such that $$F_P \simeq \mathcal F_{P, 0} \to \mathcal S_{P, 0} \simeq S_P$$
  is the Brauer--Severi scheme introduced above and $\widetilde{\mathcal F}_0$ is the Mukai flop $M$ of $F(X)$ in $(P')^\vee$.
  Using the Fano correspondence, we obtain a morphism of integral Hodge structures
  $$\varphi_{P, t} \colon H^4(\mathcal X_t, \mathbb Z) \to H^2(\mathcal F_t, \mathbb Z) \simeq H^2(\widetilde{\mathcal F}_t, \mathbb Z) \to H^2(\mathcal F_{P, t}, \mathbb Z).$$
  On the other hand, pullback induces an injective morphism of integral Hodge structures
  $$H^2(\mathcal S_{P, t}, \mathbb Z) \hookrightarrow H^2(\mathcal F_{P, t}, \mathbb Z).$$

  Since for $t \neq 0$, the cubic fourfold $\mathcal X_t$ contains only one plane, the morphism $\varphi_{P, t}$ induces an integral Hodge isometric embedding
  $$H^4(\mathcal X_t, \mathbb Z) \supset \langle h, [P] \rangle^\perp \underset{-1}\hookrightarrow H^2(\mathcal S_{P, t}, \mathbb Z)_{\mathrm{pr}}(-1),$$
  see the proof of \cite[Prop.\ 6.1.18]{huybook}.
  In fact, we have $$\varphi_{P, t}(\langle h, [P]\rangle^\perp) = \ker(2B \cdot -) \colon H^2(\mathcal S_{P, t}, \mathbb Z) \to \mathbb Z / 2 \mathbb Z),$$ where $B$ is a $B$-lift of $\alpha_{P_1}$, see \cite[Ex.\ 6.1.20]{huybook}. The claim follows by specialization to $t = 0$ and restriction to the transcendental lattices.
\end{proof}

In fact, the above can be upgraded to the level of derived categories: Recall
that
the derived category of the cubic fourfold $X$ admits a semi-orthogonal
decomposition of the form
$$D^b(X) = \langle \mathcal A_X, \mathcal O_X, \mathcal O_X(1), \mathcal O_X(2)
  \rangle,$$
where $\mathcal A_X \coloneqq \langle \mathcal O_X, \mathcal O_X(1), \mathcal
  O_X(2)
  \rangle^{\perp}$ is the Kuznetsov component of $X$, cf. \cite{kuz4fold}.

\begin{theorem}[Kuznetsov \cite{kuz4fold}, Moschetti \cite{moschetti}]
  There is a Fourier--Mukai equivalence
  $$D^b(S_P, \alpha_P) \cong \mathcal A_X.$$
\end{theorem}

Let $F(X)$ denote the Fano variety of lines on the cubic fourfold $X$. The natural morphism $F'_P \to F(X)$ is an isomorphism onto its image when restricted to the complement of the locus $P^\vee \subset F(X)$ of lines contained in the plane $P$.
We conclude this section by showing that the morphism remains injective on the whole of $F'_P$, thus allowing us to identify $F'_P$ with its image in $F(X)$ whenever the scheme structure is of no importance.

\begin{lemma}
  \label{lem:mapinjective}
  The natural morphism
  $F'_P \to F(X)$
  is injective. That is, for every line $L \subset P$, there is at most
  one $y \in \mathbb P^2$ such that $L$ lies on the residual quadric
  surface $f_P^{-1}(y)$.
\end{lemma}

\begin{proof}
  Let $L \subset P$ be a line and assume that there is $y
    \in \mathbb P^2$ such that $L \subset f_P^{-1}(y)$.  The point $y \in \mathbb P^2$ corresponds to a three-dimensional
  linear subspace $\mathbb P^3 \cong \Pi \subset \mathbb P^5$ with $\Pi \cap X
    = P \cup Q$ and $L \subset P \cap Q$. Then $\Pi \subset \bigcap_{x
      \in L} T_x X$. By \cite[Cor.\ 2.2.6]{huybook}, we have $\dim \bigcap_{x \in L} T_x X \leq
    3$. For dimension reasons, we thus have $\Pi = \bigcap_{x \in L} T_x X$, so that $\Pi$ is uniquely determined by $L$. Since $y \in \mathbb P^2$ is the image of $\Pi$ under the projection $\Bl_P \mathbb P^5 \to \mathbb P^2$, the claim follows.
\end{proof}

\begin{remark}
  As explained in \cite[Ex.\ 6.1.8]{huybook}, the intersection $F_P' \cap P^\vee \subset F(X)$ is a
  plane cubic curve. Recall from \cite[Sec.\ 2.2.2]{huybook} that lines $L \subset X$ with
  $\dim \bigcap_{x \in L} T_xX = 3$ are called lines of the second type and form a
  surface $F_2(X) \subset F(X)$. By the argument in the proof of the preceding
  lemma  we have $F_P' \cap P^\vee \subset F_2(X)$. In fact, we even have
  $$F_P' \cap P^\vee = F_2(X) \cap P^\vee \subset F(X).$$
  Indeed, let $L \subset P$ be a line and suppose that $\Pi \coloneqq
    \bigcap_{x \in L} T_x X$ is three-dimensional. As $P \subset \Pi$, we have
  $\Pi \cap X = P \cup Q$ for some quadric $Q \subset \Pi$. For every $x \in L$, we then
  have $\Pi  = T_x (\Pi \cap X)$ and thus $L \subset \Sing(\Pi \cap X) \cap P
    \subseteq Q \cap P \subset Q$.
\end{remark}

%% file: eckardt.tex
\section{Eckardt cubic fourfolds}
\label{sec:eckardt}
In order to show that the associated K3 surfaces of two planes on a cubic fourfold intersecting in a point or a line are not isomorphic,
we are going to use a degeneration argument involving Eckardt cubic fourfolds, see Theorem \ref{thm:notiso} and Theorem \ref{thm:notiso2}.  In this section, we briefly recall the definition of Eckardt cubic fourfolds and prove the properties needed to make the degeneration argument work.

Let $X$ be a smooth cubic fourfold. Recall that there is a one-dimensional
family of lines going through a general point on $X$. A point $p \in X$ is called an \emph{Eckardt point} if the
family of lines through $p$ is (at least) two-dimensional. A cubic fourfold is
called an \emph{Eckardt cubic fourfold} if it contains at least one Eckardt point.
Eckardt cubic fourfolds and their moduli have recently been studied in \cite{modpairs}. Recall the following well-known characterization of Eckardt points:
\begin{proposition}[]\label{prop:desceckardt}
  Let $X$ be a smooth cubic fourfold. The following are equivalent:
  \begin{enumerate}
    \item The point $p \in X$ is an Eckardt point;
    \item There is an involution on $X$ which fixes $p$ and a hypersurface not
          containing $p$.
    \item $X$ is projectively equivalent to a cubic fourfold given by an
          equation of the form
          $$x_0^2 l(x_1, \dots, x_5) + f(x_1, \dots, x_5)$$
          and the equivalence identifies $p$ with $[1\colon 0 \colon \dots \colon 0]$.
  \end{enumerate}
\end{proposition}

\begin{proof}
  We prove that (i) implies (iii), the other implications are clear. Let $p \in X$ be a point such that the family of lines through $p$ is two-dimensional. Pick coordinates $x_0, \dots, x_5$ on $\mathbb P^5$ such that $p = [1 \colon 0 \colon \dots \colon 0]$ and $T_p X = V(x_5) \subset \mathbb P^5$. Then, $X$ is cut out by an equation of the form
  $$x_0^2x_5 + x_0 q(x_1, \dots, x_5) + c(x_1, \dots, x_5),$$
  where $\deg q = 2$ and $\deg c = 3$. The family of lines through $p$ is parametrized by
  $$V(x_0, x_5, q, c) \subset \mathbb P^5.$$
  By assumption, the quadric $V(x_0, x_5, q)$ and the cubic $V(x_0, x_5, c)$ thus share a two-dimensional irreducible component $S \subset V(x_0, x_5) \cong \mathbb P^3.$ We claim that $S = V(x_0, x_5, c)$. Suppose that $\deg S = 1$. Then, the union of lines parametrized by $S$ form a $\mathbb P^3 \subset X$, contradicting the smoothness of $X$. Similarly, if $\deg S = 2$, then the union of lines parametrized by $S$ is a three-dimensional quadric cone. The linear span of this cone is a $\mathbb P^4$, which intersects $X$ in the cone over $S$ and a linear $\mathbb P^3$. Again, this contradicts the smoothness of $X$. All in all, it follows that $X \cap T_p X$ is a cone with vertex $p$ over a cubic surface, which is easily seen to imply (iii).
\end{proof}

Let $Y = V(f(x_1, \dots, x_5), x_0) \subset V(x_0) \cong \mathbb P^4$ and $S = V(f, l, x_0) \subset
  V(l, x_0) \cong \mathbb P^3$. Note that one can recover $X$ from the pair $(Y, V(l, x_0))$ by Proposition \ref{prop:desceckardt}.(iii). Let $L \subset S$ be one of the $27$ lines and $P
  \coloneqq \overline{pL}$. Note that we have $P \cap Y = L$. Thus there is a
commutative diagram
$$
  \begin{tikzcd}
    \Bl_L \mathbb P^4 \arrow[rd, "q"] & \Bl_L Y \arrow[r, hook] \arrow[l,
      hook'] \arrow[d, "q_Y"] & \Bl_P X
    \arrow[ld, "q_X"]  \\
    & \mathbb P^2      &
  \end{tikzcd}$$

Let $H  \coloneqq q(\Bl_L V(l)) \subset \mathbb P^2$ denote the image of the projection of $V(l)$. The
fibers of $q_X$ can be described as follows:
\begin{enumerate}
  \item Over $y \in \mathbb P^2 \setminus H$, the fiber $q_X^{-1}(y)$ is the
        double cover of $\mathbb P^2 \cong q^{-1}(y)$ branched along the conic
        $q_Y^{-1}(y)$.
  \item Over $y \in H$, the fiber $q_X^{-1}(y)$ is the cone over the conic
        $q_Y^{-1}(y)$.
\end{enumerate}

Let $C_L \subset \mathbb P^2$ and $C_P' \subset \mathbb P^2$ denote the
discriminant curve of $q_Y$ and $q_X$. Recall that $C_L$ is a quintic curve and
$C_P'$ is a sextic curve. From the above discussion, we immediately get:
\begin{lemma}\label{lem:discriminanteck}The discriminant sextic $C_P' \subset \mathbb P^2$ is the union of the quintic curve $C_L \subset \mathbb P^2$ and the line $H \subset \mathbb P^2$, i.e.,
  $$C_P' = C_L \cup H \subset \mathbb P^2.$$
\end{lemma}
The key input to the proof of Theorem \ref{thm:notisomain} is the following observation.
\begin{lemma}
  Let $L$ and $L'$ be two general lines on a cubic threefold. Then $C_L$
  and $C_{L'}$ are not isomorphic.
  \label{lem:nonisothreefold}
\end{lemma}
\begin{proof}
  Recall that the fiber of the Prym map
  $$\mathcal P \colon \mathcal R_6 \to \mathcal A_5$$
  over the intermediate Jacobian of a cubic threefold $Y$ is birational to the Fano variety of lines on $Y$, see \cite[Sec.\ 1.4. on p.\ 89]{donagistructureprym}. A general point of $\mathcal P^{-1}([IJ(Y)])$ thus corresponds to the \'etale double cover of the discriminant quintic $C_L$ induced by the conic bundle $\Bl_L Y \to \mathbb P^2$, where $L \subset Y$ is a general line. Since the smooth quintic $C_L$ admits only finitely many \'etale double covers, the fact the $\mathcal P^{-1}([IJ(Y)])$ is birational to the two-dimensional Fano variety of lines on $Y$ implies that for a general pair of lines $L, L' \subset Y$, the quintics $C_L$ and $C_{L'}$ are not isomorphic.
\end{proof}
In order to pass from the branch curves to the associated K3 surfaces, we need the following lemma.
\begin{lemma}
  Let $X$ be a very general Eckardt cubic fourfold and $P \subset X$ a plane.
  The automorphism group of the associated K3 surface $S_P$ is generated by
  the covering involution.
  \label{lem:isoeck}
\end{lemma}
\begin{proof}
  Given a general plane quintic $C$ and a general line $H$ in $\mathbb P^2$, one can construct a Eckardt cubic fourfold $X$ and a plane $P \subset X$ such that the associated discriminant sextic $C_P' \subset \mathbb P^2$ is the union of $C$ and $H$: First, pick a pair of a cubic threefold $Y \subset \mathbb P^4$ and a line $L \subset Y$ such that the quintic $C$ is the discriminant quintic of the projection $\Bl_L Y \to \mathbb P^2$. The line $H \subset \mathbb P^2$ is the image of a unique hyperplane $H' \subset \mathbb P^4$ via the projection $\Bl_L \mathbb P^4 \to \mathbb P^2$. Let $X$ denote the Eckardt cubic fourfold associated to the pair $(Y, H')$, cf.\ Proposition~\ref{prop:desceckardt}, and let $P$ denote the span of the Eckardt point and the line $L \subset Y \subset X$. Then, we have
  $$C_P' = C_L \cup H \subset \mathbb P^2$$
  by Lemma~\ref{lem:discriminanteck}.

  Hence, for a very general Eckardt cubic fourfold $X$ and a plane $P \subset X$, the associated K3 surface $S_P$ is a very general member of the family of resolutions of double covers of $\mathbb P^2$ branched along the union of a quintic and a line.
  The claim then follows from the classification of automorphism groups of K3 surfaces in
  \cite{atlasK3finite}: We have $\Pic(S_P) \simeq U(2) \oplus \mathbf{D}_4$ by \cite[Prop.\ 6.11]{atlasK3finite} and thus $\Aut(S_P) \simeq \mathbb Z /  2 \mathbb Z$ by \cite[p.\ 92]{atlasK3finite}.
\end{proof}
\begin{corollary}
  \label{lem:notisopoint}
  There exists an Eckardt cubic fourfold $X$ and two planes $P_1, P_2 \subset
    X$ intersecting in a point such that $S_{P_1}$ and $S_{P_2}$ are not
  isomorphic.
\end{corollary}

\begin{proof}
  Let $Y \subset \mathbb P^4$ be a cubic threefold. Let $L_1, L_2 \subset Y$ be
  two general lines as in Lemma \ref{lem:nonisothreefold}. Then $L_1$ and $L_2$
  span a three-dimensional linear subspace $\overline{L_1 L_2} \subset \mathbb
    P^4$. Let $S = Y \cap \overline{L_1L_2}$. Let $X \subset \mathbb P^5$ denote the associated Eckardt cubic
  fourfold with Eckardt point $p \in X$. For $i = 1, 2$, let $P_i \coloneqq
    \overline{pL_i}$ be the corresponding planes. By Lemma
  \ref{lem:discriminanteck}, we have $C_{P_1} \not \cong C_{P_2}$. Conclude by
  Lemma \ref{lem:isoeck}.
\end{proof}

Note that $P_1$ and $P_2$ intersect
in a point if and only if the lines $L_1$ and $L_2$ are disjoint. If the lines intersect, then
$P_1$ and $P_2$ intersect along a line.

\begin{corollary}
  \label{cor:notisoline}
  There exists an Eckardt cubic fourfold $X$ and two planes $P_1, P_2 \subset
    X$ intersecting along a line such that $S_{P_1}$ and $S_{P_2}$ are not
  isomorphic.
\end{corollary}

\begin{proof}
  If $P_1$ and $P_2$ are two planes intersecting in a point, then there is a
  plane $P$ such that $P_1$ and $P$, and $P_2$ and $P$ intersect in a line.
  Hence, the claim follows from the preceding corollary by contraposition.
\end{proof}

%% file: geometry.tex
\section{Cubic fourfolds containing two planes intersecting along a line}
\label{sec:geometry}
Let $X$ be a cubic fourfold containing two planes $P_1, P_2 \subset
  X$ intersecting
along a line $L = P_1 \cap P_2$.
The two planes span a three-dimensional subspace $\Pi \subset \mathbb P^5$ with $\Pi
  \cap X = P_1 \cup P_2 \cup P_3$, where $P_3$ is a third plane contained in $X$.
In the following, we assume that $X$ is a very general member of the family of
cubic fourfolds containing two planes intersecting along a line. Using the cohomological characterization of planes on a cubic fourfold \cite{voisintorelli}, see also \cite[Lem.\ 4.1, Cor.\ 4.2]{marquand}, one checks that
the cubic fourfold $X$ contains no planes other than $P_1, P_2$ and $P_3$.
\subsection{The associated K3 surfaces}
\label{sec:theassocK3}

As described in Section \ref{sec:basics}, we obtain three associated twisted K3 surfaces $(S_{P_i},
  \alpha_{P_i} \in \Br(S_{P_i})[2])$ for $i = 1, 2, 3$, which are resolutions
of double
covers of $\mathbb P^2$ branched along the plane sextics $C_{P_i}' \subset \mathbb
  P^2$. Due to the absence of planes other than $P_1, P_2$ and $P_3$, the
characterization of singular points of $C'_{P_i}$ in Lemma \ref{lem:singularitiesofcurve}
implies that the discriminant curves $C_{P_i}'$ have a unique ordinary double point each. In fact, by
\cite[Prop.\ 2.6]{stellarisingular}, we can assume that the $C_{P_i}'$ are very
general members of the family of plane sextic curves with an ordinary double
point. In particular, their normalizations $C_{P_i} \subset \Bl_{\Sing{C'_{P_i}}} \mathbb
  P^2$ are smooth irreducible curves of genus nine. Hence, the Picard groups of the K3 surfaces $S_{P_i}$ can be described as follows:
\begin{lemma}
  \label{lem:descpic}
  The
  Picard groups of the associated K3 surfaces $S_{P_i}$ are generated by the
  pullback $h \in \Pic(S_{P_i})$ of the ample generator on
  $\mathbb P^2$ and the class $s \in \Pic(S_{P_i})$ of the exceptional curve resolving
  the unique singularity of $S'_{P_i}$. With respect to this basis, the intersection form is given by
  $$h^2 = 2, s^2 = -2 \text{ and } h \cdot s = 0,$$
  and, therefore, is of discriminant $-4$.
\end{lemma}
\begin{proof}
  By construction, the sublattice $\langle h, s \rangle \subset \Pic(S_{P_i})$ of rank two admits the above description. Conversely, the surjectivity of the period map \cite{todorov}, cf.\ \cite[Thm.\ 7.4.1]{lecK3}, ensures that there exists a K3 surfaces $S'$ with $\Pic(S') = \langle h, s \rangle$, where $h$ is a big and nef class satisfying $h^2 = 2, s^2 = -2$ and $h \cdot s = 0$. Up to changing a sign, the class $s \in \Pic(S')$ represents a smooth rational curve on $S'$, cf.\ the discussion in \cite[Sec.\ 1.1.4]{lecK3}. The linear system $|h|$ then realizes $S' \to \mathbb P^2$ as the resolution of a double cover of $\mathbb P^2$ branched along a sextic with one node, cf.\ \cite[Rem.\ 2.2.4]{lecK3}. Since the $S_{P_i}$ are very general members of the family of such double covers by the discussion preceding this lemma, the claim follows.
\end{proof}

\begin{remark}\label{rem:basesid}
  The composition of the double cover $S_{P_i} \to \Bl_{\Sing(C'_{P_i})} \mathbb
    P^2$ with the projection from the unique singular point of the discriminant sextic yields an elliptic
  fibration on $S_{P_i}$ corresponding to the isotropic class $f = h -s$, which admits the following geometric description:
  Generically, points on $S_{P_i}$ parametrize rulings on quadric surfaces $Q
    \subset X$ such that $Q$ and $P_i$ span a three-dimensional linear subspace of
  $\mathbb P^5$.
  The elliptic fibration is then given by sending a ruling on a quadric $Q \subset X$ residual to $P_i$ to the hyperplane $H = \overline{Q \cup \Pi}$ spanned by $Q$ and the unique three-dimensional linear space $\Pi \subset \mathbb P^5$ containing the planes $P_1, P_2$ and $P_3$ in
  $$\Pi^\vee \coloneqq \{H \in |\mathcal O_{\mathbb P^5}(1)|
    \mid \Pi \subset H\} \cong \mathbb P^1.$$
\end{remark}

\subsection{The associated K3 surfaces are not isomorphic}

Let us begin by showing that the associated K3 surfaces $S_{P_i}$ are not isomorphic
to each other.

\begin{lemma}
  \label{lem:isok3}
  The automorphism group $\Aut(S_{P_i})  = \mathbb Z / 2 \mathbb Z$ is generated by the covering
  involution.
\end{lemma}
\begin{proof}
  This follows from the genericity assumption and the classification of
  automorphism groups of K3 surfaces, see \cite[Cor.\ 3.4]{galluzilombardopeters}.
\end{proof}

In particular, the $K3$ surfaces $S_{P_i}$ are realized in a unique way as a double cover of $\mathbb P^2$. As a crucial consequence, an isomorphism between the K3 surfaces $S_{P_1}$ and $S_{P_2}$ would induce an isomorphism between the discriminant curves $C_{P_1}'$ and $C_{P_2}'$.

\begin{theorem}
  \label{thm:notiso}
  Let $P_1, P_2 \subset X$ be a cubic fourfold containing two planes intersecting along a
  line, and assume that $X$ is very general with this property.
  Then, the associated $K3$ surfaces $S_{P_1}$ and $S_{P_2}$ are not isomorphic.
\end{theorem}
\begin{proof}
  By the preceding observation, it suffices to find a degeneration of the cubic fourfolds for which the
  associated sextics $C'_{P_1}$ and $C'_{P_2}$ are not isomorphic.
  As the
  family of cubic fourfolds with two planes intersecting along a line contains the
  family of Eckardt cubic fourfolds, this follows from Corollary \ref{cor:notisoline}.
\end{proof}

While Moschetti's result implies that the associated K3 surfaces are twisted
Fourier--Mukai partners, it turns out that this is not true for the untwisted
surfaces.

\begin{corollary}\label{cor:notderivedeq}
  Let $P_1, P_2 \subset X$ be a cubic fourfold containing two planes intersecting along a
  line, and assume that $X$ is very general with this property.
  Then, the associated K3 surfaces $S_{P_1}$ and $S_{P_2}$ are not derived equivalent.
\end{corollary}
\begin{proof}
  In fact, the K3 surface $S_{P_i}$ has no nontrivial Fourier--Mukai partner,
  see \cite[Cor.\ 5.13]{meinsmashinder}.
\end{proof}

\subsection{A correspondence}

The aim of this section is to construct a geometric correspondence between the associated K3 surfaces $S_{P_1}$ and $S_{P_2}$ in order to prove Theorem \ref{thm:exofcorr}. For the remainder of this section, we assume that $X$ is a very general member of the family of cubic fourfolds containing two planes intersecting along a line.

Let $L = P_1 \cap P_2$ denote the intersection of two of the planes. Recall that the intersection of $X$ with the linear span $\Pi$ of $P_1$ and $P_2$ is the union of three planes $P_1, P_2$ and $P_3$. Note that
the intersection of the divisors $F_{P_1}'$ and $F_{P_2}'$ in the Fano variety
splits into two
components:
$$F_{P_1}' \cap F_{P_2}' = P_3^\vee \cup F_L' \subset F(X),$$
where $P_3^\vee \subset F(X)$ is the locus of lines on the plane $P_3 \subset
  X$ and $F_L' \subset F(X)$ is the closure of the locus of lines intersecting $L$ that are not contained in
$P_1$ or $P_2$.

\begin{proposition}
  \label{lem:existenceofelliptic}
  The residual component $F_L'$ has exactly one irreducible component of dimension two and admits a morphism
  $$ F_L' \setminus \{[L]\} \to L, \;
    L' \mapsto L' \cap L,$$
  the general fiber of which is a smooth elliptic curve.
\end{proposition}

\begin{proof}
  For $x \in X$, consider the scheme $F_x \subset F(X)$ of lines through the point $x$. Note that $F_x$ is a curve unless $x \in X$ is an Eckardt point, cf.\ Section~\ref{sec:eckardt}. As we assume that $X$ is very general in the family of cubic fourfolds containing two planes intersecting along a line, $X$ does not contain an Eckardt point and thus $F_x$ is a curve for all $x \in X$.

  Picking coordinates $x_0, \dots, x_5$ on $\mathbb P^5$ such that $x = [1\, \colon 0 \,\colon 0\, \colon 0\, \colon 0\,  \colon 0]$, we can write the equation of $X$ as
  $$X = V(x_0^2x_5 + x_0q(x_1, \dots, x_5)+c(x_1, \dots, x_5)).$$
  Then, we have an isomorphism
  $$F_x \simeq V(x_5,x_0, q, c),$$
  so that $F_x$ is the complete intersection of a quadric $Q \coloneqq V(x_5, x_0, q)$ and a cubic in $\mathbb P^3$. In other words, we have $F_x \in |\mathcal O_{Q}(3)|$. For $x \in L = P_1  \cap P_2$, the loci of lines through $x$ contained in $P_1$ and $P_2$ yield two components $l_1 \cup l_2 \subset F_x$ which are lines on $\mathbb P^3$. Moreover, $l_1$ and $l_2$ intersect in the point corresponding to the line $[L]$. In the case that $Q$ is irreducible, this implies that we have $l_1 \cup l_2 \in |\mathcal O_Q(1)|$. The residual component $F'_x \subset F_x$ is thus an element of the linear system $|\mathcal O_Q(2)|$. For generic choices of the pair $(X, x)$ with the properties described above, an explicit computation with the help of \cite{sage} shows that $Q$ is irreducible and that $F_x'$ is a smooth curve. In particular, since it is a $(2, 2)$-complete intersection in $\mathbb P^3$, it is an elliptic curve.
\end{proof}

\begin{remark}
  The map $F'_L \setminus \{[L]\} \to L$ does not extend to a well-defined morphism $F_L' \to L$.
  Indeed, in the notation of the proof of the above proposition, there are exactly two points $x \in L$ for which $[L] \in F'_x$: The point $z \coloneqq P_1 \cap P_2 \cap P_3$ and another point $z' \in L$ satisfying  $$\gamma_X(z) = \gamma_X(z'),$$
  where $\gamma_X \colon X \to (\mathbb P^5)^\vee, x \mapsto T_x X \subset \mathbb P^5$ denotes the Gauss map of $X$. It is easy to verify that $[L] \in F_z'$. On the other hand, let $x \in L \setminus P_3$ be a point such that $[L] \in F_x'$. As in the proof of the above proposition, we may pick coordinates $x_0, \dots, x_5$ on $\mathbb P^5$ such that $x = [1\, \colon 0 \,\colon 0\, \colon 0\, \colon 0\,  \colon 0]$ and
  $$X = V(x_0^2x_5 + x_0q(x_1, \dots, x_5) + c(x_1, \dots, x_5)).$$
  Moreover, since $x \not \in P_3$, we can assume that $$P_1 = V(x_2, x_4, x_5), P_2 = V(x_3, x_4, x_5) \text{ and }P_3 = V(x_0, x_4, x_5).$$
  Restricted to $V(x_5)$, we can then write
  $$q|_{x_5 = 0} = x_4 l(x_1, \dots, x_4) + x_2x_3 \text{ and } c|_{x_5 = 0} = x_4 q'(x_1, \dots, x_4).$$
  The residual component $F'_x$ is then cut out by $V(x_0, x_5, q, q') \subset \mathbb P^5.$ Note that $[L]$ corresponds to the point $V(x_0, x_2, x_3, x_4 ,x_5)$. Hence, we have $[L] \in F'_x$ if and only if $q'$ does not contain the monomial $x_1^2$. One then checks that this happens if and only if the tangent space of $X$ at $[0 \, \colon 1 \, \colon 0 \,  \colon 0\, \colon 0\,  \colon 0] = P_1 \cap P_2 \cap P_3$ agrees with the tangent space at $x$, which is $V(x_5) \subset \mathbb P^5$, as claimed. Since the restriction of the Gauss map $\gamma_{X|L} \colon L \to \mathbb P^1$ is of degree two, this uniquely determines the point $z' \in L \setminus \{z\}$ with $[L] \in F'_{z'}$. A computation shows that, for general $X$ as above, the Gauss map is not ramified at $z$ and thus we have $z \neq z'$.
\end{remark}

However, the minimal model of $F_L'$ admits a well-defined elliptic fibration:

\begin{proposition}
  \label{prop:descfl}
  The projection $F'_{P_i} \to S'_{P_i}$ induces a generically finite morphism
  $F_L' \to S'_{P_i}$ of degree two.
  The minimal model $F_L$ of $F'_L$ sits in the Cartesian square
  $$
    \begin{tikzcd}
      F_L \arrow[r, "2:1"] \arrow[d]    & S_{P_i} \arrow[d] \\
      L \arrow[r, "\gamma_X|_L"] & \mathbb P^1,
    \end{tikzcd}$$
  where $\gamma_X|_L$ is the Gauss map of $X$ restricted to $L$.
  The vertical
  maps are elliptic fibrations.
\end{proposition}
\begin{proof}
  By symmetry, it is enough to consider $F'_L \to S_{P_1}'$.

  The preimage of the singular point $\Sing(S_{P_1}')$ of $S_{P_1}'$ under the
  map $F'_L \to S_{P_1}'$ is given by the locus $(P_2^\vee \cup P_3^\vee) \cap
    F_L'$ of lines on $P_2$ or $P_3$. We claim that the restriction
  $$F'_L \setminus (P_2^\vee \cup P_3^\vee) \to S_{P_1}' \setminus
    \Sing(S_{P_1}')$$
  is a finite double cover. As it is proper, it suffices to show that the
  restriction is quasi-finite. The map $F'_L \to S'_{P_1}$ is given by mapping
  a line $L'$ intersecting $L$ to its ruling on the residual quadric surface $Q
    \subset X$. Note that the residual quadric surface is uniquely determined by
  Lemma \ref{lem:mapinjective}. From now on, restrict to $S_{P_1}' \setminus
    \Sing(S_{P_1}')$. Again by Lemma \ref{lem:mapinjective}, the quadric surface $Q$
  does not contain $L$ and thus the intersection $L \cap Q$ is of length two. As $Q$
  is either a smooth quadric surface or a quadric cone, a fixed ruling of $Q$ contains at most two
  lines whose union contains the scheme-theoretic intersection $Q \cap L$.
  We conclude that the morphism $F_L' \setminus (P_2^\vee \cup
    P_3^\vee) \to S_{P_1}' \setminus \Sing(S_{P_1}')$ is finite of
  degree two. Ramification occurs precisely when $L$ intersects $Q$ in a single point.

  Recall that the Gauss map $\gamma_X \colon X \to |\mathcal O_{\mathbb P^5}(1)|$ is the map
  sending a point $x \in X$ to the tangent space $[T_xX] \in (\mathbb P^5)^\vee$.
  Restricting to the line $L \subset X$, one obtains a double cover
  $$\gamma_{X|L} \colon L \to \Pi^\vee \cong \mathbb P^1,$$
  where $\Pi \subset \mathbb P^5$ is the three-dimensional linear subspace spanned by the planes $P_1, P_2$ and $P_3$.
  One easily checks that for $L' \not \in P_2^\vee \cup P_3^\vee$, the tangent hyperplane $T_{L \cap L'} X$ is spanned by the quadric $Q$ and the two planes $P_2$ and $P_3$.

  Combining the above, it follows that the diagram
  $$
    \begin{tikzcd}
      F_L' \setminus (P_2^\vee \cup P_3^\vee) \arrow[r] \arrow[d]    & S_{P_i}' \setminus
      \Sing(S_{P_i}') \arrow[d] \\
      L \arrow[r, "\gamma_X|_L"] & \mathbb P^1
    \end{tikzcd}
  $$
  is a Cartesian square. Note that both vertical maps are elliptic fibrations.

  Therefore, $F'_L \to S'_{P_i}$ is birational to the cyclic double cover of
  $S_{P_i}$ branched along the union of the two fibers over the branch locus of the Gauss map. An explicit computation shows that, under the genericity assumptions on $X$, these two fibers are smooth elliptic curves.
  As such a double cover is a minimal surface, see, e.g., \cite[Thm.\ 3.1]{garbagnatidoublecovers}, we conclude.
\end{proof}

All in all, we obtain the following commutative diagram:
$$
  \begin{tikzcd}
    S_{P_1} \arrow[rdd, "\text{ell.\ fib.}"'] & F_L \arrow[d, "\text{ell.\ fib.}"
      description] \arrow[r, "2:1"] \arrow[l, "2:1"'] & S_{P_2} \arrow[ldd,
      "\text{ell.\ fib.}"] \\
    & L \arrow[d, "2:1" description]                                                   &                                         \\
    &\, \mathbb P^1   ,  &
  \end{tikzcd}
$$
where the map $L \to \mathbb P^1$ is the Gauss map of $X$ restricted to
$L$. In other words, there are two involutions $\iota_1, \iota_2 \in \Aut(F_L)$ with $S_{P_i} = F_L / \iota_i$ that lift the covering involution of the Gauss map $L \to \mathbb P^1$ and thus identify pairs of fibers of the elliptic fibration $F_L \to L$ in two different ways. Note that the fixed point loci of $\iota_1$ and $\iota_2$, which are the ramificattion loci of the double covers $F_L \to S_{P_i}$, coincide as they are the fibers over the ramification locus of the Gauss map $L \to \mathbb P^1$ and thus coincide for $\iota_1$ and $\iota_2$.

\begin{proposition}
  The double cover $F_L \to S_{P_i}$ splits the Brauer class $\alpha_{P_i}$,
  i.e., $f_i^* \alpha_{P_i}$ is the trivial Brauer class in $\Br(F_L)$.
\end{proposition}

\begin{proof}
  Recall that $\alpha_{P_i}$ is represented by the Brauer--Severi variety
  $F_{P_i} \to S_{P_i}$. By construction, the map $f_i \colon F_L \to S_{P_i}$
  factors (on an open subset) via a generically injective map $F_L \to
    F_{P_i}$. In particular, the base change of the Brauer--Severi variety $F_{P_i} \to S_{P_i}$ to $F_L$
  admits a rational section. It follows that the Brauer class $f_i^*
    \alpha_{P_i}$ is trivial.
\end{proof}

\begin{corollary}
  \label{cor:commonjac}
  Let $J(S_{P_i}/ \mathbb P^1) \to \mathbb P^1$ denote the relative Jacobian
  of the elliptic fibration on $S_{P_i}$.
  Then, we have isomorphisms
  $$J(S_{P_1} / \mathbb P^1) \cong J(S_{P_2} / \mathbb P^1) \cong J(S_{P_3} /
    \mathbb P^1)$$
  over $\mathbb P^1$, respecting the sections of the elliptic fibrations.
\end{corollary}
\begin{proof}
  By Remark \ref{rem:basesid}, the bases of the elliptic fibrations are naturally
  identified with the space of hyperplanes in $\mathbb P^5$ containing $\Pi$.
  As the restriction of the Gauss map to $L = P_1 \cap P_2$ is
  independent of the chosen plane $P_1$ and $P_2$, the elliptic fibrations $S_{P_1} \to
    \mathbb P^1 \leftarrow S_{P_2}$ are fiberwise isomorphic by Proposition
  \ref{prop:descfl}. By construction, the same holds for the associated Jacobian K3 surfaces $J(S_{P_i} / \mathbb P^1) \to \mathbb P^1$. In particular, the discriminant of the elliptic K3 surfaces $J(S_{P_i}) \to \mathbb P^1$, i.e., the locus of points on $\mathbb P^1$ over which the elliptic fibrations are singular, are naturally identified. Since the elliptic K3 surfaces $S_{P_1}$ and $S_{P_2}$ are very general in their corresponding moduli space of lattice polarized K3 surfaces by the discussion in Section~\ref{sec:theassocK3}, the claim follows by
  \cite[Thm.\ 1.1]{reconstructionellipticK3}.
\end{proof}

\subsection{The ramification curves}

\label{sec:recillas}

Before continuing the discussion of the relation between the two associated K3
surfaces, let us digress on a geometric relation between the branch
curves $C_{P_i} \subset \Bl_{\Sing(C'_{P_i})} \mathbb P^2$.
In fact, they are related by Recillas' trigonal construction:
Let $X \to \mathbb P^1$ be a tetragonal curve of genus $g-1$. Let $\tilde{C}
  \coloneqq S^2_{\mathbb P^1} X \to \mathbb P^1$ be the relative second symmetric
product. As $X \to \mathbb P^1$ is of degree four, there is a natural involution on
$\tilde{C}$ given by swapping residual pairs. Let $C$ denote the quotient of
$\tilde{C}$ by this involution.
\begin{theorem}[{\cite{recillas}, \cite[Sec.\ 2]{donagisfibersprym}}]
  \label{thm:rec}
  The above construction yields a bijection
  $$\left\{
    \begin{array}{c}\text{Tetragonal curves } \\ X\text{ of genus }g-1
    \end{array}\right\} \congpf
    \left\{
    \begin{array}{c}\text{ Trigonal curves }C\text{ of} \\ \text{genus }g\text{ with a double
        cover }\tilde{C}
    \end{array}\right \}$$
\end{theorem}

In our situation, the composition $C_{P_i} \subset S_{P_i} \to \mathbb P^1$ is
of degree four. Hence, the $C_{P_i}$ are tetragonal curves, whose genus is nine as they are normalizations of plane sextics with exactly one node. Let $C \subset S$
denote the curve of nontrivial two-torsion points in the Jacobian elliptic K3
surface $S \coloneqq J(S_{P_i} / \mathbb P^1)$, which coincide for $S_{P_1}$ and $S_{P_2}$ by Corollary \ref{cor:commonjac}. The composition $C \subset S \to \mathbb P^1$ is of degree three. By \cite[Thm.\ 7.6]{vangeemenremarks}, the curve $C$ is of genus ten and there is an isomorphism
$$J(C)[2] \cong \Br(S)[2].$$
In particular, there is an \'etale double cover $\tilde{C}_i \to C$ corresponding to the class
$$[S_{P_i}] \in \Sha(S)[2] \cong \Br(S)[2] \cong J(C)[2] \cong \{\text{\'etale double covers of } C\} / \text{isomorphism}.$$
\begin{proposition}
  Via Recillas' trigonal construction, i.e., Theorem \ref{thm:rec}, with $g = 10$, the tetragonal curve $C_{P_i}$
  corresponds to the double cover $\tilde{C}_i \to C$.
\end{proposition}
\begin{proof}
  This is explained in \cite[Sec.\ 8.6]{vangeemenremarks}.
\end{proof}

In particular, the tetragonal curves $C_{P_1}$ and $C_{P_2}$ are both related to the same trigonal curve $C$ via Recillas' trigonal construction.

\subsection{Invariants of the surface \texorpdfstring{$F_L$}{FL}}

Let us conclude this section by computing invariants of the surface $F_L$. As before, we assume that $X$ is a very general member of the family of cubic fourfolds containing two planes intersecting along a line.

\begin{proposition}
  \label{prop:invariants}
  The surface $F_L$ is of Kodaira dimension one and satisfies
  $$q(F_L) = 0 \text{ and } p_g(F_L)= 3.$$
\end{proposition}
\begin{proof}
  This follows immediately from the description of $F_L$ as the double cover of an
  elliptic K3 surface branched along the disjoint union of two smooth fibers, cf.\ the proof of Proposition~\ref{prop:descfl}.
  See \cite[Thm.\ 3.1]{garbagnatidoublecovers}.
\end{proof}

\subsection{Lines on a cubic fourfold intersecting a given line}
\label{sec:involutions}
While the remainder of this note only relies on the geometry of lines intersecting the special line $L = P_1 \cap P_2$, we briefly discuss the geometry of the locus $\mathfrak F_L \subset F(X)$ of lines intersecting an arbitrary line $L \subset X$ on a cubic fourfold $X$. Note that we have
$$\mathfrak F_L = F_L' \cup P_1^\vee \cup P_2^\vee,$$
in the special case that the line $L$ is the intersection of two planes $P_1$ and $P_2$ on $X$ as above.
On the other hand, the case of generic lines $L \subset X$ on a smooth cubic fourfold $X$ has been studied by Huybrechts in
\cite{huybrechtslines}. For $[L'] \in  \mathfrak F_L$ such that the plane $\overline{L L'}$ is not contained in $X$, there is
a residual line $L'' \subset X$ in the intersection
$$\overline{L L'} \cap X = L \cup L' \cup L''.$$
The map $L' \mapsto L''$ yields a (rational) involution $\iota_L$ on $\mathfrak
  F_L$. The quotient $\mathfrak F_L \to \mathfrak F_L / \iota_L \cong \mathfrak D_L \subset \mathbb P^3$ is realized by
viewing $\mathfrak F_L$ as the relative Fano variety of lines over the
discriminant locus $\mathfrak D_L \subset \mathbb P^3$ of the conic
fibration $\Bl_L X \to \mathbb P^3$. For general $L \subset X$, the discriminant
locus $\mathfrak D_L \subset \mathbb P^3$ is a nodal quintic
surface with $p_g(\mathfrak D_L) = 4$, whereas $\mathfrak F_L$ is a smooth irreducible surface with $p_g(\mathfrak F_L) = 5$, see \cite[Thm.\ 0.1, Thm.\ 0.2]{huybrechtslines}. Thus,
the action of $\iota_L$ on $H^{2, 0}(\mathfrak F_L)$ decomposes as the direct sum
of a one-dimensional $\iota_L$-anti-invariant subspace and a four-dimensional $\iota_L$-invariant subspace.
Moreover, Huybrechts establishes a Hodge isometry
$$(H^4(X, \mathbb Z)_{\mathrm{pr}},-(\,.\;)) \cong (H^2(\mathfrak F_L, \mathbb
  Z)_{\mathrm{pr}}^-, (1/2)(\,.\;))$$
for general $L \subset X$, see \cite[Thm.\ 0.2]{huybrechtslines}.
On the right-hand side, $H^2(\mathfrak F_L, \mathbb Z)_{\mathrm{pr}}$ denotes
the primitive part with respect to the restriction of the Plücker embedding and
$H^2(\mathfrak F_L, \mathbb Z)^-_{\mathrm{ pr}} \subset H^2(\mathfrak F_L, \mathbb
  Z)_{\mathrm{pr}}$ is its $\iota_L$-anti-invariant part.

From now on, assume that $X$ contains a plane $P \subset X$ and that $L \subset P$. Then, we have
$$\mathfrak F_L = \mathfrak F_{L, P} \cup P^\vee \subset F(X),$$
where $\mathfrak F_{L, P}$ is the closure of the locus of lines intersecting $L$
which are not contained in $P$.

\begin{figure}[ht!]
  \centering
  \input{sketch.tex}
  \caption{The (rational) involutions on \texorpdfstring{$\mathfrak F_{L, P}$}{FLP}}
  \label{fig:sketchinvolution}
\end{figure}

Since $L \subset P$, we can (rationally) embed $\mathfrak F_{L, P}$ into the Brauer--Severi variety $F_P'$ over the (singular) K3 surface $S_P'$. By the same arguments as in the proof of Proposition \ref{prop:descfl}, the composition
$$\mathfrak F_{L, P} \dashrightarrow F_P' \to S_P'$$
is generically finite of degree two: Indeed, a point $[(\Pi, R)] \in S_P'$ corresponds to a linear threespace $\Pi \subset \mathbb P^5$ containing $P$ and a
ruling $R$ on the residual quadric surface $Q$ in $\Pi \cap X = P \cup Q$. If $L$ is not contained in $Q$, then the intersection $Q \cap L$ is of length two,
and the two lines in the ruling $R$ containing $Q \cap L$ are the only lines in $\mathfrak F_{L, P}$ mapping to $(\Pi, R)$.

Let $\iota_{L, P}$ denote the (rational) covering involution of $\mathfrak F_{L, P} \to S_P'$.
One easily checks that the involutions $\iota_{L, P}$ and $\iota_L$ commute. The geometry of the situation is depicted in Figure \ref{fig:sketchinvolution}.

Specializing to the case of $L = P_1 \cap P_2$, we have
$$\mathfrak F_L = F_L' \cup P_1^\vee \cup P_2^\vee \text{ and } \mathfrak F_{L, P_1} = F_{L}' \cup P_2^\vee.$$
By the construction above, we obtain the two involutions $\iota_{L, P_1}, \iota_{L, P_2}$ and the rational involution $\iota_L$ on the surface $F_L$.
As argued above, both $\iota_{L, P_1}$ and $\iota_{L, P_2}$ commute with $\iota_L$ and $F_L / \iota_{L, P_i} \cong S_{P_i}$.
One can also consider the quotient $F_L / \iota_L$: As above, there is a rational map $F_L \dashrightarrow \mathbb P^3, L' \mapsto \overline{LL'}$,
which (birationally) realizes the quotient $F_L / \iota_L$ as the discriminant quintic surface $\mathfrak D_L \subset \mathbb P^3$ of the conic bundle $\Bl_L X \to \mathbb P^3$.
Let $D_L$ denote a resolution of singularities of $\mathfrak D_L$.

\begin{lemma}\label{lem:pgdl}
  Let $X \subset \mathbb P^5$ be a general member of the family of cubic fourfolds containing two planes $P_1, P_2 \subset X$ intersecting along a line $L$. Then, we have $$p_g(D_L) = 2$$
\end{lemma}

\begin{proof}
  Pick coordinates on $\mathbb P^5$ such that $$P_1 = V(x_2, x_4, x_5), P_2 = V(x_3, x_4, x_5) \text{ and } P_3 = V(x_1, x_4, x_5).$$ Then, the discriminant surface $\mathfrak D_L \subset \mathbb P^3$ is cut out by the determinant of a matrix of the form
  $$M \coloneqq \begin{pmatrix}
      l_{00}(y_4, y_5)        & l_{01}(y_4, y_5)                 & q_0(y_2, y_3, y_4, y_5)          \\
      l_{01}(y_4, y_5)        & l_{11}(y_4, y_5)                 & q_1(y_2, y_3, y_4, y_5) + y_2y_3 \\
      q_0(y_2, y_3, y_4, y_5) & q_1(y_2, y_3, y_4, y_5) + y_2y_3 & f(y_2, y_3, y_4, y_5)
    \end{pmatrix},$$
  where $\deg l_{ij} = 1$, $\deg q_i = 2$, $\deg f = 3$ and $$q_0(y_2, y_3, 0, 0) = q_1(y_2, y_3, 0, 0) = f(y_2, y_3, 0, 0) = 0.$$
  The singularities of $\mathfrak D_L$ can be described as follows: For $i = 1,2$, let $p_i \in \mathbb P^3$ denote the point corresponding to the plane $P_i$. A local computation shows that, under our genericity assumption, the points $p_i \in \mathfrak D_L$ are ordinary singularities of multiplicity three. Note that we have $M(p_1) = M(p_2) = 0.$

  We claim that there are exactly eight points $p_3, \dots, p_{10}$ such that $M(p_i)$ has rank $1$: In this case, the matrix
  $$N \coloneqq \begin{pmatrix} l_{00}(y_4, y_5) & l_{01}(y_4, y_5) \\ l_{01}(y_4, y_5) & l_{11}(y_4, y_5) \end{pmatrix}$$
  has rank at most $1$. The vanishing locus of $\det N$ is the union of two planes $P'_1, P'_2 \subset \mathbb P^3$. Then, there exist $[a \colon b] \in \mathbb P^1$ such that after restricting to $P'_1$ we have $a l_{00} = b l_{01}$ and $a l_{01} = b l_{11}$. The locus of points on $P'_1$ such that $M$ is of rank at most $1$ is then cut out by the two equations
  $$aq_0 = b(q_1+y_2y_3) \text{ and } q_0^2 = fl_{00}.$$
  Counted with multiplicities, we obtain eight points
  $$2p_1, 2p_2, p_3, p_4, p_5, p_6,$$
  over which $M$ has rank at most $1$. The same argument applied to $P_2'$ yields four more points $p_7, \dots, p_{10}$ over which $M$ has rank $1$. A local computation shows that $p_3, \dots, p_{10} \in \mathfrak D_L$ are ordinary double points. By explicit computation, e.g., using \cite{sage}, one checks that, under our genericity assumptions, the singular locus of the surface $\mathfrak D_L$ coincides with $\{p_1, \dots, p_{10}\}$.

  Since the singularities are ordinary, we conclude
  \[p_g(D_L) = p_g(\mathfrak D_L) - \sum_{x \in \Sing(\mathfrak D_L)} (\operatorname{mult}_x(\mathfrak D_L) -2)  = 4 - 2 = 2 \qedhere \]
\end{proof}
The fact that $$p_g(F_L) = 3 = p_g(D_L) + p_g(S_{P_i})$$ will be revisited in Remark \ref{rem:pgdecomp}.

%% file: sketch.tex
\begin{tikzpicture}
  \filldraw[black, dashed, thick] (1.7,-1) -- (4, -3.0) node[below, node
  font=\tiny] {$(\iota_{L, P} \circ \iota_L)(L')$};
  \filldraw[black, dashed, thick] (4.34,-1) -- (2, -3.0) node[below, node
  font=\tiny] {$\iota_{L, P}(L')$};
  \draw[gray] (0, 0) -- (2, -2) -- (6, -2) -- (4, 0) -- cycle;
  \draw[gray] (1, -1) -- (5, -1);
  \draw[gray] (2.2, -2) .. controls (0.6, 0.5) and (2.5, 0.5) .. (5.8, -2);
  \filldraw[black] (1.7,-1) circle (1pt);
  \filldraw[black] (4.34,-1) circle (1pt);
  \filldraw[black, thick] (1.7,-1) -- (4, 1.0) node[above, node font=\tiny] {$L'$};
  \filldraw[black, thick] (4.34,-1) -- (2, 1.0) node[above, node font=\tiny] {$\iota_{L}(L')$};

  \filldraw[black] (3.015,0.14) circle (1pt);
  \filldraw[black] (3.015,-2.14) circle (1pt);

  \node[node font=\tiny] at (0.5, -0.2) (P) {$P$};
  \node[node font=\tiny] at (4.8, -1.7) (PQ) {$P \cap Q$};
  \node[node font=\tiny] at (3, -1.15) (L) {$L$};
\end{tikzpicture}

%% file: transcendental.tex
\section{Transcendental lattices}
\label{sec:transcendental}

The aim of this section is to understand the action of the correspondence described in the previous section on the transcendental
cohomology of the associated K3 surfaces.
As before, let $X$ be a cubic fourfold containing two planes $P_1, P_2 \subset X$ intersecting along a line $L = P_1 \cap P_2$, and assume that $X$ is very general with this property.
As discussed in the previous section, we obtain two double covers
$f_i \colon F_L \to S_{P_i}$. Recall that $p_g(F_L) = h^{2, 0}(F_L) = 3$, see Proposition \ref{prop:invariants}. The observation that $f_1$
and $f_2$ have identical ramification loci immediately yields the following:

\begin{proposition}
  The pullbacks of the holomorphic two-forms on the K3 surfaces $S_{P_i}$ to $F_L$ coincide, i.e., we have
  $f_1^* H^{2, 0}(S_{P_1}) = f_2^* H^{2, 0}(S_{P_2}) \subset H^{2,
        0}(F_L)$.
\end{proposition}
\begin{proof}
  Let $\omega_j \in H^{2, 0}(S_{P_j})$ be a symplectic form. As the ramification loci of
  $f_1$ and $f_2$ coincide by the description in Proposition \ref{prop:descfl}, we have
  $$V(f_1^* \omega_1) = \Ram(f_1) = \Ram(f_2) = V(f_2^* \omega_2) \in
    |\omega_{F_L}|.$$
  The claim follows.
\end{proof}

As a consequence, the correspondence induces a rational Hodge isometry between
the transcendental parts of the rational cohomology groups of $S_{P_1}$ and
$S_{P_2}$.

\begin{corollary}
  \label{cor:rationalhodge}
  The pullbacks of the rational transcendental components of $$H^2(S_{P_1}, \mathbb
    Q) \text{ and }H^2(S_{P_2}, \mathbb Q)$$ to $F_L$ coincide, i.e., $$f_1^* T(S_{P_1}) \otimes \mathbb Q =
    f_2^* T(S_{P_2}) \otimes \mathbb Q \subset H^2(F_L, \mathbb Q).$$
  In particular, the correspondence $S_{P_1} \leftarrow F_L \rightarrow S_{P_2}$ induces a rational Hodge isometry
  $$f_{2, *} \circ f_1^* \colon T(S_{P_1}) \otimes \mathbb Q \congpf T(S_{P_2}) \otimes \mathbb Q.$$
\end{corollary}

\begin{proof}
  By \cite[Lem.\ 3.2.7]{lecK3}, both $f_1^* T(S_{P_1}) \otimes \mathbb Q$ and $f_2^* T(S_{P_2}) \otimes \mathbb Q$ are irreducible rational Hodge structures generated by $f_1^* H^{2, 0}(S_{P_1}) = f_2^* H^{2, 0}(S_{P_2})$. The claim follows.
\end{proof}

\begin{remark}
  In fact, one can show that
  $$f_1^*H^2(S_{P_1}, \mathbb Q) \cap f_2^* H^2(S_{P_2}, \mathbb Q) =  \mathbb Q \langle f
    \rangle \oplus
    f_{i}^*T(S_{P_{i}}) \otimes \mathbb Q,$$
  where $f \in \NS(F_L)$ is the class of a fibre of the elliptic fibration on
  $F_L$.
\end{remark}

Next, we consider the integral structure. First, note that Corollary \ref{cor:rationalhodge} can not be upgraded to an integral Hodge isometry between the transcendental lattices of $S_{P_1}$ and $S_{P_2}$.

\begin{lemma}
  \label{lem:transcendental_not_iso}
  The lattices $T(S_{P_1})$ and $T(S_{P_2})$ are not Hodge isometric.
\end{lemma}
\begin{proof}
  This follows from Corollary \ref{cor:notderivedeq} and the derived Torelli
  theorem for K3 surfaces due to Mukai \cite{mukai} and Orlov \cite{orlov}, see also \cite[Cor.\ 16.3.7]{lecK3}.
\end{proof}

Let $B \in H^2(S_{P_1}, \frac{1}{2} \mathbb Z) \subset H^2(S_{P_1}, \mathbb Q)$ be
a $B$-field lifting the Brauer class $\alpha_{P_1} \in \Br(S_{P_1})[2]$, i.e, a preimage of $\alpha_{P_1}$ via the map
$$H^2(S_{P_1}, \frac{1}{2} \mathbb Z) \to H^2(S_{P_1}, \mathbb Q/ \mathbb Z) \to \Br(S_{P_1})$$
induced by the exponential sequence. Note that $B$ is well-defined up to an element
in $$H^2(S_{P_1}, \mathbb Z) + \frac{1}{2} \Pic(S_{P_1}).$$

\begin{lemma}
  \label{lem:residueclasses}
  Any $B$-field lift $B \in H^2(S_{P_1}, \frac{1}{2} \mathbb Z)$ of $\alpha_{P_1}$ satisfies
  $$2B^2 \equiv 2Bh \equiv 2Bs \equiv 1 \mod 2.$$
  In particular, these residue classes are independent of the choice of $B$-field
  representing the Brauer class $\alpha_{P_1} \in \Br(S_{P_1}).$
\end{lemma}
\begin{proof}
  One checks that the above residue classes mod $2$ do not change when we replace $B$ by
  $$B' = B + u + \frac{k h + l s}{2}.$$
  In view of \cite[Lem.\ 6.4]{kuz4fold}, it remains to show $2Bs \equiv 1 \mod 2$.

  By applying the arguments in the proof of
  \cite[Lem.\ 6.2]{kuz4fold} verbatim to the rational curve $C \subset S_{P_1}$ representing $s \in \Pic(S_{P_1})$ instead of a curve in the linear system $|h|$,
  we get
  $$2Bs \equiv \deg \det E \mod 2,$$
  where $E$ is a rank two vector bundle on $C$ such that $F_{P_1}|_C \simeq \mathbb P(E)$.
  The claim then follows as the restriction of the Brauer--Severi
  variety $F_{P_1}$ to the rational curve $\mathbb P^1 \simeq C$ is isomorphic to the projection
  $$\Bl_{\mathrm{pt}} \mathbb P^2 \simeq \mathbb P_{\mathbb P^1}(\mathcal O_{\mathbb P^1} \oplus \mathcal O_{\mathbb P}(1))\to \mathbb P^1$$ by Lemma
  \ref{lem:desc_resolution}.
\end{proof}
\begin{corollary}
  The Brauer class $\alpha_{P_1} \in \Br(S_{P_1})$ is nontrivial.
\end{corollary}
\begin{proof}
  If $\alpha_{P_1} \in \Br(S_{P_1})$ were trivial, we could choose $B = 0$ as a $B$-field lift, which would satisfy $B^2 = 0$ in contradiction to the previous lemma.
\end{proof}

Recall that the Fano correspondence induces a Hodge isometry
$$(T(X), -(\,.\,)) \cong T(S_{P_1}, B),$$
cf.\ Proposition \ref{prop:fanohodge}.
Since $T(S_{P_1}, B) \subset T(S_{P_1})$ is a sublattice of index two and by Lemma \ref{lem:descpic}, we have $$\disc T(X) =
  \disc T(S_{P_1}, B) = 4 \disc T(S_{P_1}) = 16.$$

Note that, as the cover $f_i \colon F_L \to S_{P_i}$ is of degree two, pullback induces a Hodge isometry
$$f_i^* \colon (T(S_{P_i}), (\,.\,)) \congpf (f_i^* T(S_{P_i}) \subset H^2(F_L, \mathbb Z), 1/2(\,.\,)).$$

\begin{theorem}
  \label{thm:intersectionistx}
  The restriction of the Fano correspondence to the transcendental lattice of the cubic fourfold $X$ induces a Hodge isometry
  $$T(S_{P_1}, \alpha_{P_1}) \cong (T(X), -(\,.\,)) \cong (f_1^* T(S_{P_1}) \cap f_2^* T(S_{P_2}), 1/2(\,.\,)).$$
\end{theorem}
\begin{proof}
  As $T(S_{P_1}) \not \cong T(S_{P_2})$ by Lemma \ref{lem:transcendental_not_iso},
  we have \begin{equation}\label{eq:thm47_1}f_1^* T(S_{P_1}) \cap f_2^* T(S_{P_2}) \subsetneq f_1^*
    T(S_{P_1}).\end{equation}
  Since $f_1^* T(S_{P_1}) \otimes \mathbb Q = f_2^* T(S_{P_2}) \otimes \mathbb Q$, the inclusion (\ref{eq:thm47_1}) is of finite index. Thus, when equipped with the intersection form $1/2(\,.\,)$, the discriminant of the left-hand side of (\ref{eq:thm47_1}) is at least $16 = 2^2 \disc T(S_{P_1})$.
  On the other hand, the Fano correspondence induces a Hodge isometric
  embedding
  \begin{equation}\label{eq:thm47_2}(T(X), -(\,.\,)) \hookrightarrow (f_1^* T(S_{P_1}) \cap f_2^* T(S_{P_2}), 1/2(\,.\,)).\end{equation}
  As the left-hand side of (\ref{eq:thm47_2}) has discriminant $16$, while the right-hand side has discriminant at least $16$, the claim follows.
\end{proof}

Recall from Corollary~\ref{cor:commonjac}, that the Jacobian elliptic K3 surfaces associated to the elliptic fibrations $S_{P_1} \to \mathbb P^1$ and $S_{P_2} \to \mathbb P^1$ coincide. We denote the common Jacobian elliptic fibration of $S_{P_1} / \mathbb P^1$ and $S_{P_2} / \mathbb P^1$ by $S \to \mathbb P^1$.
One can view $S_{P_1}$ and $S_{P_2}$ as Tate--\v{S}avarevi\v{c} twists of $S$ and thus, see \cite[Ch.\
  11]{lecK3}, there are Brauer classes $$\beta_i \coloneqq [S_{P_i}] \in \Sha(S) \cong
  \Br(S).$$
Furthermore, there is a natural restriction homomorphism $$r_i \colon \Br(S) \simeq \Hom(T(S), \mathbb Q / \mathbb Z) \to \Hom(T(S_{P_i}),  \mathbb Q / \mathbb Z) \simeq
  \Br(S_{P_i}),$$ induced by the inclusion $T(S_{P_i}) \subset T(S)$. Note that we have $\ker(r_i) = \langle \beta_i \rangle \subset \Br(S).$
\begin{theorem}
  \label{thm:descbrauer}
  The Brauer classes on the Jacobian K3 surface $S$ corresponding to the Tate--\v{S}afarevi\v{c} twists $S_{P_1}$ and $S_{P_2}$ restrict to the Brauer classes on the K3 surfaces $S_{P_1}$ and $S_{P_2}$ induced by the geometry of the cubic fourfold, i.e.,
  $$r_1(\beta_2) = \alpha_{P_1} \in \Br(S_{P_1}) \text{ and }\;
    r_2(\beta_1) = \alpha_{P_2} \in \Br(S_{P_2}).$$
\end{theorem}

In order to prove the theorem, we need to understand the possible overlattices of $T(X)$, which are controlled by its discriminant group.
Instead of trying to compute the discriminant group of the lattice $T(X)$ directly, we use the twisted Mukai lattice of the twisted K3 surface $(S_{P_1}, \alpha_{P_1})$ and the Hodge isometry $T(X) \cong T(S_{P_1}, \alpha_{P_1})$ together with the fact that, up to a sign, the discriminant group of a primitive sublattice of a unimodular lattice is canonically isomorphic to the discriminant group of its orthogonal complement, which turns out to be easier to control in our situation.
For the definition and basic properties of twisted Mukai lattices, consult the
original \cite{huystellari}.
We let $\tilde{H}(S_{P_1}, B, \mathbb Z)$ denote the $B$-twisted Mukai
lattice of $S_{P_1}$. Recall that, on the level of abelian groups, we have
$$\tilde{H}(S_{P_1}, B, \mathbb Z) = H^0(S_{P_1}, \mathbb Z) \oplus H^2(S_{P_1}, \mathbb Z) \oplus H^4(S_{P_1}, \mathbb Z).$$
Fix  generators $e_0, e_4 \in \tilde{H}(S_{P_1}, B, \mathbb Z)$ of $H^0(S_{P_1},
  \mathbb Z), H^4(S_{P_1}, \mathbb Z) \subset \tilde{H}(S_{P_1}, B, \mathbb Z)$.

\begin{lemma}[Cf.\ {\cite[Lem.\ 3.1]{macristellari}}]
  \label{lem:picdesc}
  The algebraic part of the twisted Mukai lattice is given by
  $$\widetilde{H}^{1,1}(S_{P_1}, B, \mathbb Z) = \langle 2e_0+2B, \Pic(S_{P_1}),e_4\rangle.$$
\end{lemma}
\begin{proof}
  The right-hand side is easily seen to be contained in the left-hand side. Since both have
  discriminant $16$, the claim follows.
\end{proof}

\begin{lemma}
  \label{lem:discTx}
  The discriminant group of the lattice $T(X)$ is
  $$A_{T(X)} \coloneqq T(X)^\vee / T(X) = \mathbb Z / 4 \mathbb Z \oplus \mathbb Z
    / 4 \mathbb Z.$$ The discriminant form is given by the intersection matrix
  $$
    \begin{pmatrix} \frac{1}{2} & \frac{3}{4} \\ \frac{3}{4} &
                \frac{1}{2}
    \end{pmatrix}.$$
\end{lemma}

\begin{proof}
  As $\tilde{H}(S_{P_1}, B, \mathbb Z)$ is unimodular and
  $$(T(X), -(\, . \,)) \cong T(S_{P_1}, B) = \tilde{H}^{1, 1}(S_{P_1}, B, \mathbb Z)^{\perp} \subset \tilde{H}(S_{P_1}, B, \mathbb Z),$$
  we have $$(A_{T(X)}, q) \cong (A_{T(S_{P_1}, B)}, -q) \cong (A_{\widetilde{H}^{1,1}(S_{P_1}, B, \mathbb Z)}, q),$$
  where $q$ denotes the respective discriminant form.
  The claim then follows from Lemma \ref{lem:picdesc} and Lemma \ref{lem:residueclasses} by a straightforward computation.
\end{proof}

\begin{corollary}
  \label{cor:descoverlattice}
  There is a unique overlattice $T(X) \subset T'$ of index four, up to isomorphism. Furthermore,
  we have
  $$T' / T(X) \cong \mathbb Z / 2 \mathbb Z \oplus \mathbb Z / 2 \mathbb Z.$$
\end{corollary}
\begin{proof}
  This follows from the correspondence between overlattices of $T(X)$ and isotropic subgroups of the discriminant group $A_{T(X)}$, see \cite[Prop.\ 1.5.1]{nikulin}, cf.\ \cite[Prop.\ 14.0.2]{lecK3}.
\end{proof}

In the following, we give two proofs of Theorem \ref{thm:descbrauer}. The first one is lattice-theoretic and uses the description of discriminant group of $T(X)$ from above. The second one is more geometric and uses the relative Jacobian of the elliptic fibration $F_L \to \mathbb P^1$.

\begin{proof}[Proof of Theorem \ref{thm:descbrauer}]
  Composing the embeddings $$T(X) \cong T(S_{P_i}, \alpha_{P_i}) \hookrightarrow T(S_{P_i}) \hookrightarrow T(S)$$ realizes $T(S)$ as an overlattice of $T(X)$ of index $4$, which is unique by Corollary \ref{cor:descoverlattice}.
  In particular, we have the following commutative diagram, where each arrow represents a
  Hodge isometric embedding of index two.
  $$
    \begin{tikzcd}
      & T(S)                                                                                 &                                  \\
      T(S_{P_1}) \arrow[ru] & T(S_{P_2}) \arrow[u]                                                      & T(S_{P_3}) \arrow[lu] \\
      & T(X) \arrow[ru] \arrow[u] \arrow[lu] &
    \end{tikzcd}
  $$
  The inclusions $T({S_{P_i}})\subset T(S)$ correspond to the Brauer classes $\beta_i \in \Br(S)$, i.e., $T(S_{P_i}) = T(S, \beta_i) \subset T(S),$ which are distinct by Lemma \ref{lem:transcendental_not_iso}, while the inclusion $T(X) \subset T(S_{P_i})$ corresponds to the Brauer class $\alpha_{P_i} \in \Br(S_{P_i})$, i.e., $T(X) = T(S_{P_i}, \alpha_{P_i}) \subset T(S_{P_i})$.
  The theorem then follows from the fact that $T(S) / T(X) \cong (\mathbb Z / 2 \mathbb Z)^{\oplus 2}$ by Corollary \ref{cor:descoverlattice} and the commutativity of the diagram above. More precisely, we can argue as follows: By construction, we have $r_i(\beta_j) = \beta_j |_{T(S_{P_i})}$. Therefore, we have
  $$T(S_{P_i}, r_i(\beta_j)) = \ker(\beta_j|_{T(S_{P_i})}) = T(S_{P_i}) \cap T(S_{P_j}),$$
  since $T(S_{P_j})$ is the kernel of $\beta_j$. Then, $T(X)$ is contained in $\ker(\beta_j|_{T_{S_{P_i}}})$, but it also has index $2$ in $T(S_{P_i})$, so we must have
  $$T(S_{P_i}, r_i(\beta_j)) = \ker(\beta_j|_{T(S_{P_i})}) = T(S_{P_i}) \cap T(S_{P_j}) = T(X).$$
  Thus, $\alpha_i$ and $r_i(\beta_j)$ have the same kernel and therefore, they are proportional. But since $\alpha_i$ has order $2$ in $\Br(S_{P_i})$, this means that $\alpha_i = r_{i}(\beta_j)$.

  More geometrically, one could argue as follows: The (twisted) Poincar\'e sheaf on $F_L \times_{\mathbb P^1} J(F_L/\mathbb P^1)$ induces a
  morphism $\phi \colon H^2(F_L, \mathbb Z) \to H^2(J(F_L / \mathbb P^1), \mathbb
    Z),$ where $J(F_L / \mathbb P^1)$ denotes the relative Jacobian of the elliptic fibration $F_L \to \mathbb P^1$.
  Moreover, taking the quotient by the induced action of $\iota_{1} \in \Aut(F_L)$ on $J(F_L / \mathbb P^1)$, which agrees with the one induced by $\iota_2 \in \Aut(F_L)$, yields a finite morphism $f \colon J(F_L / \mathbb P^1) \to S$ of degree two.
  One can check that $\phi$ induces a commutative diagram of the form
  $$
    \begin{tikzcd}
      & f^*T(S) && \subset H^2(J(F_L / \mathbb P^1), \mathbb Z)                                                                                                        \\
      f_1^*T(S_{P_1}) \arrow[ru] & f_2^*T(S_{P_2}) \arrow[u]
      & f_3^*T(S_{P_3}) \arrow[lu]
      & \subset H^2(F_L, \mathbb Z) \arrow[u,
        "\phi"]\\
      & T(X) \arrow[ru]
      \arrow[u] \arrow[lu] &
      & \subset H^4(X, \mathbb Z) \arrow[u, "\Phi"],
    \end{tikzcd}
  $$
  which immediately yields the claim.
\end{proof}

As a consequence of Corollary \ref{cor:descoverlattice}, we obtain the following relation:
\begin{corollary}
  The Brauer classes $\beta_i \in \Br(S)$ satisfy
  $$[S_{P_1}]\cdot[S_{P_2}] = \beta_1 \cdot \beta_2 = \beta_3 = [S_{P_3}] \in \Br(S).\footnote{Despite the identification $\Br(S) \simeq \Hom(T(S), \mathbb Q / \mathbb Z)$ via the exponential sequence, we write the group law of the Brauer group multiplicatively.}$$
\end{corollary}
\begin{proof}
  As $T(S)$ is an overlattice of $T(X)$ of index four, we have $T(S) / T(X) \cong (\mathbb Z / 2 \mathbb Z)^{\oplus 2}$ by Corollary \ref{cor:descoverlattice}.
  As the nontrivial elements in $T(S) / T(X)$ correspond to the three distinct Brauer classes $\beta_i \in \Br(S)$ by the arguments in the proof of Theorem \ref{thm:descbrauer}, the claim follows.
\end{proof}

\begin{remark}
  \label{rem:pgdecomp}
  As discussed towards the end of Section~\ref{sec:involutions},
  in addition to the covering involutions $\iota_{1}, \iota_{2} \in \Aut(F_L)$ with quotient $f_i \colon F_L \to F_L / \iota_{i} \cong S_{P_i}$
  there is another (rational) involution $\iota_L$ on $F_L$ whose quotient $f_L \colon F_L \dashrightarrow F_L / \iota_L$ is birational to a (very singular) quintic surface $\mathfrak D_L \subset \mathbb P^3$. Let $D_L$ denote a minimal resolution of $\mathfrak D_L$.
  By Lemma \ref{lem:pgdl}, we have $p_g(D_L) = 2$. A specialization of the arguments in \cite[Sec.\ 3.2]{huybrechtslines} shows that $\iota_L$ acts as $-1$ on the one-dimensional subspace $$f_i^* H^{2, 0}(S_{P_i}) \subset H^{2, 0}(F_L),$$
  while it clearly acts as $+1$ on the two-dimensional subspace $$f_L^* H^{2, 0}(D_L) \subset H^{2, 0}(F_L).$$
  Since $p_g(F_L) = 3$, it follows that
  $$H^{2, 0}(F_L) = f_1^* H^{2, 0}(S_{P_1}) \oplus f_L^* H^{2, 0}(D_L),$$
  thus yielding a decomposition of the (rational) transcendental lattice of $F_L$ into two orthogonal components:
  $$T(F_L) \otimes \mathbb Q = f_1^* T(S_{P_1}) \otimes \mathbb Q \oplus f_L^* T(D_L) \otimes \mathbb Q.$$
\end{remark}

%% file: derived.tex
\section{Derived categories}

\label{sec:derived}

The aim of this section is to study the relation between the derived categories
of the associated twisted K3 surfaces and their double cover $F_L$.

\subsection{Twisted derived categories of elliptic K3 surfaces}\label{sec:derived2}
Recall from Corollary~\ref{cor:commonjac} that both $S_{P_1}$ and $S_{P_2}$ are Tate--\v{S}afarevi\v{c} twists of the Jacobian elliptic K3 surface $S$.
Thus, we obtain Brauer classes $\beta_i = [S_{P_i}] \in \Br(S)[2]$. As before, we let
$$r_i \colon \Br(S) \simeq \Hom(T(S), \mathbb Q / \mathbb Z) \to \Hom(T(S_{P_i}), \mathbb Q / \mathbb Z) \simeq \Br(S_{P_i})$$
denote the restriction map, cf. \cite[Thm.\ 5.4.3]{caldararuthesis}.
By Theorem \ref{thm:descbrauer}, we have
\begin{align*}
  r_i(\beta_i) = 0 \in \Br(S_{P_i})
  \text{ and }
  r_i(\beta_j) = \alpha_{P_i} \in \Br(S_{P_i})
\end{align*}
for $1\leq i \neq j \leq 3$. For $\alpha \in \Br(S) \cong \Sha(S)$, let $S_{\alpha}$ denote
the associated elliptic K3 surface and $r_{\alpha} \colon \Br(S) \to \Br(S_{\alpha})$ the natural restriction map. In particular, we have $S_{\beta_i} = S_{P_i}.$ By a result of Donagi and Pantev, Tate--\v{S}afarevi\v{c} twists of $S$ are twisted derived equivalent:
\begin{theorem}[{Donagi--Pantev \cite[Thm.\ A]{donagipantev}}]
  Let $S \to \mathbb P^1$ be an elliptic K3 surface with a section (and at worst $I_1$-fibers) and
  $\alpha, \beta \in \Br(S)$. Then there is an exact linear equivalence
  $$D^b(S_{\alpha}, r_{\alpha}(\beta)) \cong D^b(S_{\beta},r_{\beta}(\alpha)^{-1}).$$
\end{theorem}
We thus have two explanations for the equivalence between $D^b(S_{P_1}, \alpha_{P_1})$ and $D^b(S_{P_2}, \alpha_{P_2})$: Applying \cite[Thm.\ 1.2]{moschetti} twice yields the equivalence
$$D^b(S_{P_1}, \alpha_{P_1}) \simeq \mathcal A_X \simeq D^b(S_{P_2}, \alpha_{P_2}),$$
while combining Theorem~\ref{thm:descbrauer} and \cite[Thm.\ A]{donagipantev} yields the equivalenc
$$D^b(S_{P_1}, \alpha_{P_1}) \simeq D^b(S_{\beta_1}, r_1(\beta_2)) \simeq D^b(S_{\beta_2}, r_2(\beta_1)) \simeq D^b(S_{P_2}, \alpha_{P_2}).$$

It would be interesting to know the precise relation between the equivalence constructed using the theorem of Donagi and Pantev
and the equivalence via the Kuznetsov component. As a small step towards a better understanding, the action of the equivalence via the Kuznetsov component on the twisted Mukai
lattices is studied in Appendix~\ref{sec:latticecomp}.

\subsection{Twisted derived categories of cyclic double covers}

In order to relate the twisted derived categories of the two K3 surfaces $S_{P_1}$ and the surface $F_L$,
we study equivariant twisted derived categories of double covers. In particular, we generalize the results of \cite[Sec.\ 4]{kuzperry} to the twisted setting.

Let $Y$ be an algebraic variety and $\mathcal L$ a line bundle on $Y$. Suppose
$Z$ is a Cartier divisor in $Y$ defined by a section of $\mathcal L^2$. Let $f
  \colon X \to Y$ be the double cover of $Y$ ramified over $Z$, let $\iota \in
  \Aut(X)$ denote the covering involution and fix a Brauer class $\alpha \in \Br(Y)$.

\begin{proposition}
  \label{prop:nat_twisted_action}
  The involution $\iota \in \Aut(X)$ naturally acts on $D^b(X, f^*\alpha)$ via pullback.
\end{proposition}

\begin{proof}
  Choose an Azumaya algebra $\mathcal A$ representing $\alpha$. Note that we have the
  natural identification $\iota^*
    f^* \mathcal A = f^* \mathcal A$ and thus a well-defined action
  $$\iota^* \colon \Coh_{f^* \mathcal A}(X) \congpf \Coh_{\iota^*f^* \mathcal A}(X) =
    \Coh_{f^* \mathcal A}(X)$$
  and thus also
  $$\iota^* \colon D^b(X, f^*\alpha) = D^b(\Coh_{f^* \mathcal A}(X)) \congpf
    D^b(\Coh_{f^* \mathcal A}(X)) =  D^b(X, f^* \alpha).$$
  If instead of $\mathcal A$ we started with an equivalent Azumaya algebra
  $\mathcal B \cong \mathcal A \otimes \End(F)$, where $F$ is a vector bundle on
  $Y$, representing the Brauer class
  $\alpha$ on $Y$, then $E \mapsto E \otimes f^* F^\vee$ induces an equivalence
  $$\Coh_{f^* \mathcal A}(X) \congpf \Coh_{f^* \mathcal B}(X)$$
  that fits into the commutative diagram,
  $$
    \begin{tikzcd}
      \Coh_{f^* \mathcal A}(X) \arrow[r, "\iota^*"] \arrow[d, "\otimes f^*F^\vee"] &
      \Coh_{f^* \mathcal A}(X) \arrow[d, "\otimes f^*F^\vee"] \\
      \Coh_{f^*\mathcal B}(X) \arrow[r, "\iota^*"]                               &
      \Coh_{f^*\mathcal B}(X) .
    \end{tikzcd}$$
  Thus, the derived equivalence is independent of the choice of Azumaya algebra.
\end{proof}

\begin{remark}
  Let $X$ be an algebraic variety equipped with an automorphism $f \in
    \Aut(X)$ and a Brauer class $\alpha \in \Br(X)$. In order to get a
  natural action of $\langle f \rangle \subset \Aut(X)$ on $D^b(X, \alpha)$ via pullback, it is not enough to have $f^*
    \alpha = \alpha$ on the level of Brauer classes as there is no preferred
  choice of an equivalence
  $$D^b(X, f^* \alpha) \cong D^b(X, \alpha)$$
  in general.\footnote{Thanks to Ziqi Liu for discussions regarding this remark.}
\end{remark}

Denoting by $j' \colon Z \hookrightarrow Y$ and $j \colon Z \hookrightarrow X$
the embeddings of $Z$ as the branch and ramification divisor, we have a
commutative diagram:
$$
  \begin{tikzcd}
    & X \arrow[d, "f"] \\
    Z \arrow[r, "j'", hook] \arrow[ru, "j", hook] & Y
  \end{tikzcd}$$
As in  \cite[Sec.\ 4.1]{kuzperry}, there are functors
\begin{align*}
  f^*    & \colon D^b(Y, \alpha) \to D^b(X, f^* \alpha)^{\langle \iota^* \rangle} \\
  j_{ *} & \colon D^b(Z, (f \circ j)^* \alpha) \to D^b(X, f^* \alpha)^{\langle
      \iota^*
      \rangle}.
\end{align*}

\begin{theorem}\label{thm:twistedsod}
  Both $f^*$ and $j_{*}$ are fully faithful.
  Moreover, there is a semiorthogonal decomposition of the form
  $$D^b(X, f^* \alpha)^{\langle \iota^* \rangle} = \langle f^* D^b(Y, \alpha),\, j_*
    D^b(Z, (f \circ j)^* \alpha)\rangle.$$
\end{theorem}
\begin{proof}
  This follows verbatim from the arguments given in \cite[Thm.\ 4.1]{kuzperry}, replacing
  sheaves by twisted sheaves.
\end{proof}

Before applying the above to our situation, let us give a couple of remarks regarding the derived category of an Enriques surface.
\begin{example}
  Let $S$ be an Enriques surface and $f\colon \tilde S \to S$ its K3 cover. There is a unique nontrivial Brauer class $\alpha \in \Br(S)$, see e.g., \cite{beauvillenriques}.
  By Theorem \ref{thm:twistedsod}, we have
  $$D^b(S, \alpha) \cong D^b(\tilde S, f^{*} \alpha)^{\mathbb Z / 2 \mathbb Z}.$$
\end{example}

\begin{remark}
  \label{rem:brauertriv}
  Now suppose we have $f \colon X \to Y$ as above with $f^* \alpha = 1$. Choose an
  Azumaya algebra $\mathcal A$ on $Y$ representing $\alpha \in \Br(Y)$. Triviality of $f^* \alpha$ implies that
  there is a vector bundle $F$ on $X$ such that $f^* \mathcal A \cong \End(F)$ and a
  Morita--equivalence
  \begin{align*}
    \Theta_{F} \colon D^b(X, f^*\mathcal A) & \congpf D^b(X)                         \\
    G                                       & \mapsto F^* \otimes_{f^* \mathcal A} G
  \end{align*}
  By conjugation, we obtain $\Phi_{\iota, F} \coloneqq \Theta_F \circ \iota^* \circ \Theta_F^{-1} \in \Aut(D^b(X))$.
\end{remark}

\begin{example}
  Let $S$ be an Enriques surface and $f \colon \tilde{S} \to S$ its K3 cover.
  Let $\iota \in \Aut(S)$ denote the covering involution and let $\alpha \in \Br(S)$ be the
  unique nontrivial Brauer class on $S$. As the involution acts
  freely, we have
  $$D^b(S) \cong D^b(\tilde{S})^{\langle \iota^* \rangle} \text{ and } D^b(S,
    \alpha) \cong D^b(\tilde{S}, f^* \alpha)^{\langle \iota^* \rangle}$$
  by Theorem \ref{thm:twistedsod}.

  Beauville \cite{beauvillenriques} has shown that there are divisors in the moduli
  space of Enriques surfaces on which one has $f^* \alpha = 1$.
  In this case, there is a line bundle $L \in \Pic(\tilde{S})$ satisfying $\iota^* L = L^{-1}$.
  The autoequivalence described in Remark \ref{rem:brauertriv}
  has already been studied by Reede in \cite{reede2024enriques}.
  One can choose an Azumaya algebra $\mathcal A$ on $S$ such that $f^* \mathcal A = \End(\mathcal
    O_X \oplus L)$, i.e., $F = \mathcal O_X \oplus L$ and check that
  $$\Phi_{\iota, F}(E) = \iota^* E \otimes L,$$
  see \cite{reede2024enriques}.
  Then we have
  $$D^b(S, \alpha) \cong D^b(\tilde{S}, f^* \alpha)^{\langle \iota^* \rangle} \cong
    D^b(\tilde{S})^{\langle\Phi_{\iota, F}\rangle}.$$

\end{example}

Let us now apply the above to the case of cubic fourfolds containing two planes
intersecting along a line.
As before, let $X$ be a cubic fourfold containing two planes $P_1, P_2 \subset X$ intersecting along a line $L = P_1 \cap P_2$, and assume that $X$ is very general with this property.
Let $\iota_i \in \Aut(F_L)$ denote the covering involution of $f_i
  \colon F_L \to S_{P_i}$ for $i = 1, 2$ and let $j \colon \mathbb E \subset F_L$ denote the ramification locus of the double covering
$F_L \to S_{P_i}$. Recall that $\mathbb E$ is the disjoint union of two smooth
elliptic curves.

By \cite[Thm.\ 4.1]{kuzperry}, we have the following semiorthogonal
decomposition:

\begin{lemma}
  There are semiorthogonal decompositions
  $$D^b(F_L)^{\langle \iota_i^* \rangle} = \langle D^b(S_{P_i}), D^b(\mathbb E) \rangle.$$
\end{lemma}

As $f_i^* \alpha_{P_i}$ is trivial, we can find vector bundles $F_i$ on
$F_L$ and Azumaya algebras $\mathcal A_i$ on $S_{P_i}$ such that
$$f_i^* \mathcal A_i \cong \End(F_i).$$
By the considerations in Remark \ref{rem:brauertriv}, we obtain thus autoequivalences
$$\Phi_i \coloneqq \Theta_{F_i} \circ \iota_i^* \circ \Theta_{F_i}^{-1}
  \in \Aut(D^b(F_L)).$$
\begin{lemma}
  There are semiorthogonal decompositions
  $$D^b(F_L)^{\langle \Phi_i \rangle} = \langle f_i^* D^b(S_{P_i}, \alpha_{P_i}),
    D^b(\mathbb E) \rangle.$$
\end{lemma}
\begin{proof}
  Note that we have $(f \circ j)^*\alpha = 1$ as $\Br(\mathbb E)$ is trivial since $\mathbb E$ is a curve. The statement then follows from Theorem
  \ref{thm:twistedsod} and conjugation by $\Theta_{F_i}$.
\end{proof}

We are now ready to prove the main result of this section.

\begin{theorem}
  \label{thm:derivedmain}
  There is a linear exact equivalence
  $$D^b(F_L)^{\langle \Phi_1 \rangle} \cong D^b(F_L)^{\langle \Phi_2 \rangle},$$
  which respects the semiorthogonal decompositions
  $$D^b(F_L)^{\langle \Phi_i \rangle} = \langle f_i^* D^b(S_{P_i}, \alpha_{P_i}),
    D^b(\mathbb E) \rangle.$$
\end{theorem}
\begin{proof}
  By combining the results of \cite{donagipantev} and \cite{stellariFM}, there is an $(\alpha^{-1}_{P_1} \boxtimes
    \alpha_{P_2})$-twisted Fourier--Mukai kernel $$\mathcal F \in D^b(S_{P_1}
    \times S_{P_2}, \alpha^{-1}_{P_1} \boxtimes \alpha_{P_2})$$ inducing the equivalence
  $$D^b(S_{P_1}, \alpha_{P_1}) \congpf D^b(S_{P_2}, \alpha_{P_2}),$$
  which is linear with respect to the elliptic fibrations $S_{P_i} \to \mathbb
    P^1$. The pullback of this
  kernel under the cover $F_L \times F_L \to S_{P_1} \times S_{P_2}$
  yields an $(\Phi_1 \times \Phi_2)$-equivariant Fourier--Mukai kernel, inducing a functor
  $$\langle f_1^* D^b(S_{P_1}, \alpha_{P_1}),
    D^b(\mathbb E) \rangle = D^b(F_L)^{\langle \Phi_1 \rangle} \to D^b(F_L)^{\langle \Phi_2 \rangle} = \langle
    f_2^* D^b(S_{P_2}, \alpha_{P_2}),
    D^b(\mathbb E) \rangle.$$ By construction, it induces an
  equivalence between $f_1^* D^b(S_{P_1}, \alpha_{P_1})$ and $f_2^*
    D^b(S_{P_2}, \alpha_{P_2})$. As the Donagi--Pantev
  equivalence \'etale locally on the base is given by the Poincare sheaf, we see
  that it induces the classical Fourier--Mukai equivalence on $D^b(\mathbb E)$.
\end{proof}

%% file: others.tex
\section{Disjoint and pointwise intersection}

\label{sec:others}

In this section, we briefly discuss the other two cases: Two planes which are
disjoint or intersect in a point. As mentioned in the introduction, the three cases each yield $18$-dimensional families of smooth cubic fourfolds and, therefore, do not specialize to each other.

\subsection{Disjoint planes}

Let $X$ be a cubic fourfold containing two disjoint planes $P_1, P_2 \subset X$.
As noted before, in this case, the associated K3 surfaces are isomorphic.
\begin{theorem}[{\cite[§3, App.]{voisintorelli}}]
  There is an isomorphism
  $$S_{P_1} \cong S_{P_2}$$
  \label{lem:isodisjoint}and the Brauer classes $\alpha_{P_i} \in \Br(S_{P_i})$ are trivial.
\end{theorem}
\begin{proof}
  Let $F_{P_1, P_2} = F'_{P_1} \cap F'_{P_2} \subset F(X)$ denote the locus of lines intersecting both $P_1$ and
  $P_2$. In particular, we get morphisms $F_{P_1, P_2} \to S_{P_1}$ and $F_{P_1, P_2} \to S_{P_2}$
  which one shows to be isomorphisms by geometric considerations similar to
  the ones in Section \ref{sec:geometry}, see \cite[§3, App.]{voisintorelli}.

  Moreover, the inclusion $S_{P_1} \cong F_{P_1, P_2} \subset F'_{P_1}$ yields a
  section of the Brauer--Severi variety $F_{P_1}  / S_{P_1}$, hence the Brauer class
  $\alpha_{P_1}$ is trivial in $\Br(S_{P_1})$.
\end{proof}

\begin{remark}
  The isomorphism $S_{P_1} \cong S_{P_2}$ does not respect the double
  covers. In fact, the automorphism group
  $$\Aut(S_{P_1}) = \mathbb Z / 2 \mathbb Z * \mathbb Z
    / 2\mathbb Z$$
  is freely generated by the two covering involutions, see \cite[Cor.\ 3.4]{galluzilombardopeters}.
\end{remark}

\subsection{Pointwise intersection}

Let $X$ be a cubic fourfold containing two planes $P_1$ and $P_2$ intersecting
in a point, and assume that $X$ is very general with this property. In particular, we assume that there are no planes on $X$ other than $P_1$ and $P_2$. In this case, the ramification curves $C_{P_i} \subset \mathbb P^2$
are smooth sextics admitting a tritangent by the characterization in Lemma \ref{lem:singularitiesofcurve}.
The preimage of the tritangent under the double cover splits into two smooth rational curves. Using the same strategy as in Lemma \ref{lem:descpic}, one can describe the Picard lattice of the K3 surfaces as follows:
\begin{lemma}\label{lem:descpic2}
  The Picard groups of the associated K3 surfaces $S_{P_i}$ are generated by the classes $s_1, s_2 \in \Pic(S_{P_i})$ of the smooth rational curves covering the tritangent to the sextic in $\mathbb P^2$. With respect to this basis, the intersection form is given by
  $$s_1^2 = s_2^2 = -2 \text{ and } s_1 \cdot s_2 = 3.$$
\end{lemma}
As in the case of planes intersecting along a line, there is a unique involution on the associated K3 surfaces:
\begin{lemma}
  \label{lem:isoK3p}
  The automorphism group $\Aut(S_{P_1}) \cong \mathbb Z / 2 \mathbb Z$ is
  generated by the covering involution.
\end{lemma}
\begin{proof}
  This follows from the classification of
  automorphisms of K3 surfaces with prescribed Picard lattice, see
  \cite[Cor.\ 3.4]{galluzilombardopeters}.
\end{proof}

Again, this implies that any isomorphism between the K3 surfaces $S_{P_1}$ and $S_{P_2}$ would induce an isomorphism between the ramification curves $C_{P_1}'$ and $C_{P_2}'$.

\begin{theorem}
  \label{thm:notiso2}
  The K3 surfaces $S_{P_1}$ and $S_{P_2}$ are not isomorphic.\footnote{It was communicated to us by Reinder Meinsma that one can use the description of the Picard lattices in Lemma~\ref{lem:descpic2} to deduce (with the help of a computer algebra system) that the K3 surfaces $S_{P_i}$ admit no non-trivial Fourier--Mukai partners. In particular, $S_{P_1}$ and $S_{P_2}$ are not derived equivalent.}
\end{theorem}
\begin{proof}
  This follows from Lemma \ref{lem:notisopoint} and Lemma \ref{lem:isoK3p}, cf.\ the proof of Theorem~\ref{thm:notiso}.
\end{proof}

\subsection{Open questions}

Comparing the three cases, it is natural to raise the following questions:

\begin{question}
  In the case of disjoint or one-dimensional intersection, we found a geometric explanation
  of the twisted equivalence
  $$D^b(S_{P_1}, \alpha_{P_1}) \cong D^b(S_{P_2}, \alpha_{P_2})$$
  without passing via the Kuznetsov component. Is
  there a similar ``geometric'' way of producing such an equivalence in the case of
  two planes intersecting in a point?
\end{question}

\begin{question}
  In the case where the two planes intersect along a line, one can relate the two
  ramification curves $C_{P_i} \subset \mathbb P^2$ via Recillas' construction,
  cf.\ Section  \ref{sec:recillas}. Is there an analogous geometric relation in
  the other two cases?
\end{question}

%% file: latticecomp.tex
\section{Induced maps between twisted Mukai lattices}
\label{sec:latticecomp}
In this section, we study the action of the twisted derived equivalences on the
algebraic parts of the twisted Mukai lattices of the K3 surfaces associated to
cubic fourfolds containing two planes intersecting along a line. This is motivated by the question of figuring out the precise relation between the equivalence via the Kuznetsov component and the Donagi--Pantev equivalence, cf.\ Section \ref{sec:derived2}. At this point however, we have been unable to compute the action of the Donagi--Pantev equivalence on the twisted Mukai lattice in a way that would allow us to compare it to the equivalence via the Kuznetsov component.

We will use the notation introduced in Section \ref{sec:transcendental}. Let $X$ be a cubic fourfold containing a plane $P \subset X$, and assume that $X$ is very general with this property. Then $\Pic(S_P) = \langle h \rangle$, where $h$ is the pullback of the ample
generator under $S_P \to \mathbb P^2$. We
can pick a $B$-field
$B \in H^2(S_P, \mathbb Q)$ lifting the Brauer class $\alpha_P \in \Br(S_P)$.
As in Lemma \ref{lem:picdesc}, one easily checks that
$$\tilde{H}^{1, 1}(S_P, B, \mathbb Z) = \langle 2e_0 + 2B, h, e_4\rangle.$$
By \cite[Cor.\ 8]{addington}, cf.\ \cite[Cor.\ 6.5.1]{huybrechtsK3cat}, there is $\lambda_1
  \in \tilde{H}(\mathcal A_X, \mathbb Z)$ such
that
$$H^2(F(X), \mathbb Z) \cong \lambda_1^{\perp} \subset \tilde{H}(\mathcal A_X,
  \mathbb Z).$$
By the arguments in \cite[Sec.\ 3]{macristellari}, one can compose the Kuznetsov
equivalence
$$\mathcal A_X \cong D^b(S_P, \alpha_P)$$
with twists by line bundles such that it induces a Hodge isometry
\begin{align*}
  \tilde{H}(\mathcal A_X, \mathbb Z) & \congpf \tilde{H}(S_P, B, \mathbb Z) \\
  \lambda_1                          & \mapsto (0, h, 0).
\end{align*}
Note that by \cite[Thm.\ 3.19]{yoshioka}, there is a Hodge isometry
$$\tilde{H}(S_P, B, \mathbb Z) \supset (0, h, 0)^\perp \cong H^2(M_{S_{P}, \alpha_P}(0, h, 0), \mathbb Z),$$
where $M_{S_{P}, \alpha_P}(0, h, 0)$ is the moduli space of $\alpha_P$-twisted sheaves
on $S_P$ with Mukai vector $(0, h, 0)$.
In particular, we get a commutative square
$$
  \begin{tikzcd}
    \tilde{H}(\mathcal A_X, \mathbb Z) \arrow[r, "\sim"] & \tilde{H}(S_P, B, \mathbb Z) \\
    H^2(F(X), \mathbb Z) \cong \lambda_1^{\perp}\arrow[r, "\sim"] \arrow[u, hook] &
    (0, h, 0)^{\perp} \cong H^2(M_{S_{P}, \alpha_{P}}(0, h, 0),
    \mathbb Z) \arrow[u, hook].
  \end{tikzcd}.$$
In fact, the lower horizontal map is induced by a birational map
$$F(X) \xdashrightarrow{\,\sim\,} M_{S_{P}, \alpha_{P}}(0, h, 0),$$
cf.\ \cite[Sec.\ 3]{macristellari}.
The Picard group of the Fano variety $F(X)$ is generated by the classes
$$\Pic(F(X)) = \langle g, [F_P]\rangle,$$
where $g$ is the ample class inducing the Plücker embedding. There is a rational
Lagrangian fibration on $F(X)$ corresponding to the isotropic class $g - [F_P]$.

\begin{lemma}\label{lem:deschodge1}
  On the algebraic part, the Hodge isometry above is given by
  \begin{align*} \langle g, [F_P] \rangle = \Pic(F(X)) & \congpf h^{\perp} \cap \tilde{H}^{1,
               1}(S_P, B, \mathbb Z) = \langle h' , e_4 \rangle                             \\
               g                                     & \mapsto  h' + k_4 e_4,               \\
               [F_P]                                 & \mapsto h'+(k_4-1)e_4
  \end{align*}
  where $h' \coloneqq (B \cdot h) h - 4e_0 - 4B$ and $k_4 \coloneqq \frac{6-(h')^2}{8} \in
    \mathbb Z$.
\end{lemma}
\begin{proof}
  By the uniquess of the rational Lagrangian fibration on $F(X)$, cf.\ \cite[Ex.\ 1.4.3, p.\ 24]{brookethesis}, the isotropic class
  $g - [F_P]$
  has to be mapped to the isotropic class $e_4$ which induces the unique Lagrangian
  fibration on $M_{S_P, \alpha_P}(0, h, 0)$. Let $w \in h^\perp \cap \tilde{H}^{1,
      1}(S_P, B, \mathbb Z)$ denote the image of $g$ under the Hodge isometry
  above. Then we have $w^2 = g^2 = 6$
  and $w \cdot e_4 = g \cdot (g-[F_P]) = 4$. Solving the resulting equations,
  we obtain $w = h' + k_4 e_4$ with $k_4 \in \mathbb Z$ as above.
\end{proof}

Now, we specialize to the situation where the cubic fourfold contains two planes
intersecting along a line. When constructing a degeneration of the associated K3
surfaces over a small disk such that the cohomology groups assemble into a trivial
local system, one has to make a choice (Note that the moduli space of marked K3
surfaces is not separated).

In our situation, we can construct such a family as
follows, using a trick that was already applied by Moschetti \cite{moschetti}:
Let $X$ be a very general cubic fourfold containing two planes $P_1, P_2$
meeting along a
line $L$. Let $P_3$ denote the residual plane. Embed $X$ as a hyperplane
section in a general cubic fivefold $Y \subset \mathbb
  P^6$. Let $F(\Bl_{P_1} Y / \mathbb P^3) \to \mathcal S' \to \mathbb P^3$
denote the Stein factorization of the relative Fano variety of the quadric
surface fibration $\Bl_{P_1} Y \to \mathbb P^3$, which is ramified along a
(singular) sextic surface. Note that $S_{P_1}' \subset \mathcal S'$ and by
varying hyperplane sections of $Y$ containing $P_1$, we obtain a family of cubic
fourfolds containing a plane, a special member being $X$, while the general member
contains precisely one plane.

In order to resolve the threefold singularity of
$\mathcal S'$ corresponding to the singular point of $S'_{P_1}$, which turns out
the be an ordinary double point, one has to make a choice. As explained in
\cite{kuzsing}, this geometrically amounts to choosing one of the planes $P_2$
and $P_3$ and then performing a flip in its dual plane on the relative Fano variety. For details,
we refer to \cite{kuzsing}.

For the purpose of the following computation, let us
choose $P_3$ and then carry out the construction explained in \cite{kuzsing}. We
obtain a family of cubic fourfolds $P_1 \times \Delta \subset \mathcal X \to \Delta$ over the disk, the general member of which contains precisely one plane, while the special
fiber  $X = \mathcal X_0$ contains three planes $P_1, P_2, P_3$ spanning a
$\mathbb P^3$. Additionally, the construction yields a family $\mathcal S \to
  \Delta$ of smooth K3 surfaces and a family of Brauer--Severi varieties $\mathcal
  F \to \mathcal S \to \Delta$, such that $(\mathcal S_t, [\mathcal F_t] \in \Br(S_t))$ is the
twisted K3 surface associated to $P_1 \subset X_t$. Here, the Brauer--Severi
variety $\mathcal F_{0} \to S_0 = S_{P_1}$ is constructed by performing a flip
of $F'_{P_1}$ in $P_{3}^\vee \subset F'_{P_1}$. In particular, over the
singular point of $S'_{P_1}$, the Brauer--Severi variety restricts to the
projection
$$\Bl_{P_2^\vee \cap P_3^\vee (= \mathrm{pt})} P_2^\vee \to \mathbb P^1,$$
cf. Lemma \ref{lem:desc_resolution}.
\begin{lemma}
  On the algebraic part, the Hodge isometry above is given by
  \begin{align*} \langle g, [F_{P_1}], [F_{P_2}] \rangle = \Pic(F(X)) & \congpf h^{\perp} \cap \tilde{H}^{1,
               1}(S_{P_1}, B_1, \mathbb Z) = \langle h' ,s,  e_4 \rangle                                      \\
               g                                                    & \mapsto  h' + k_4 e_4,                  \\
               [F_{P_1}]                                            & \mapsto h'+(k_4-1)e_4                   \\
               [F_{P_2}]                                            & \mapsto -s + \frac{1-k_s}{2} e_4        \\
               \Bigl( [F_{P_3}]                                     & \mapsto s + \frac{1+k_s}{2} e_4 \Bigr),
  \end{align*}
  where $k_s \coloneqq B \cdot s \equiv 1 \mod 2$.
\end{lemma}

\begin{proof}
  The description of the images of $g$ and $[F_{P_1}]$ follows from Lemma
  \ref{lem:deschodge1} by specialization. It remains to compute the images of
  $[F_{P_j}]$ for $j = 2, 3$. Again, comparing the intersection products, one
  easily verifies that
  $$[F_{P_j}] \mapsto \pm s + \frac{1 \pm k_s}{2} e_4.$$
  It remains to check which sign is realized by $[F_{P_2}]$. Let $\Sigma
    \subset F_{P_1} \to S_{P_1} \to S'_{P_1}$ denote the fibre over the singular
  point in $S'_{P_1}$. By our convention, we have
  $$\Sigma = \Bl_{P_2^\vee \cap P_3^\vee = (\mathrm{pt})} P_2^\vee.$$
  In particular, we can write $\Pic(\Sigma) =
    \langle H, E \rangle$, where $H$ is the pullback of the ample generator on
  $P_2^\vee$ and $E$ is the class of the exceptional divisor. Note that we
  have $g|_{\Sigma} = H +E.$ The intersection $F_{P_2}' \cap P_2^\vee$ is a
  cubic curve in $P_2^\vee$, see \cite[Ex.\ 6.1.8]{huybook}. Combined with the
  observation that the restriction of $[F_{P_j}]$ to $F_{P_1} \to S_{P_1}$ has
  fiber degree two, it follows that we have \begin{equation}\label{eq:app1}[F_{P_2}]|_{\Sigma} = 3H-E.\end{equation}
  Since the canonical bundle of (the Mukai flop of) $F(X)$ is trivial, the adjunction formula yields $[F_{P_1}]|_{F_{P_1}} = [\omega_{F_P}]$. As $\Sigma \subset F_{P_1} \to S_{P_1}'$ is the fiber over the singular point, the adjunction formula for $\Sigma \subset F_{P_1}$ then yields
  \begin{equation}\label{eq:app2}[F_{P_1}]|_{\Sigma} = [\omega_{F_{P_1}}]|_{\Sigma} = [\omega_{\Sigma}]|_{\Sigma}-[\mathcal O_{F_{P_1}}(\Sigma)]|_{\Sigma} = (E-3H) - 2(E-H) = -g|_{\Sigma}.\end{equation}
  By combining Equation (\ref{eq:app1}) and Equation (\ref{eq:app2}) with $g = [F_{P_1}] + [F_{P_2}] + [F_{P_2}]$, we obtain
  $$[F_{P_3}]|_\Sigma = 3E-H.$$
  Thus we get
  \begin{align*}\pm 2s^2 = \langle([F_{P_2}]-[F_{P_3}]) ,  s \rangle_{\tilde{H}(S_{P_1},
    B_1, \mathbb Z)} & = \frac{1}{2}\int_{F_{P_1}}( [F_{P_2}] -
    [F_{P_3}])g|_{F_{P_1}} \pi^* s                              \\ &= \frac{1}{2}\int_{\Sigma} (4H-4E)(H+E)
    = 4.
  \end{align*}
  As $s^2 = -2$, we conclude that the Hodge isometry sends
  $[F_{P_2}]$ to $-s + \frac{1-k_s}{2} e_4.$
\end{proof}

\begin{corollary}
  On the algebraic part, the Hodge isometry induced by composing Kuznetsov's
  equivalences is given by
  \begin{align*}
    \langle 2e_0+2B_1, h, s, e_4 \rangle = \tilde{H}^{1, 1}(S_{P_1}, B_1, \mathbb Z) & \congpf \tilde{H}^{1,
    1}(S_{P_2}, B_2, \mathbb Z) = \langle 2e_0+2B_2, h, s, e_4 \rangle                                                                                     \\
    h                                                                                & \mapsto h                                                           \\
    2e_0+2B_1                                                                        & \mapsto 2e_0+2B_2 - \frac{1}{2}k_4(e_4 - ( h' + (k_4-1)e_4))        \\
    s                                                                                & \mapsto \frac{1-k_s}{2}(s+\frac{1+k_s}{2}
    e_4)-\frac{1+k_s}{2}(h'+(k_4-1)e_4)                                                                                                                    \\
    e_4                                                                              & \mapsto -4e_0-4B_2 + (B_2 \cdot h) h + s + \frac{2k_4+k_s-1}{2}e_4.
  \end{align*}
  In particular, one can view $S_{P_1}$ as a moduli space of $\alpha_{P_2}$-twisted
  sheaves of Mukai vector $$4e_0+4B_2 - (B_2 \cdot h) h - s - \frac{2k_4+k_s-1}{2}e_4
    \in \tilde{H}^{1, 1}(S_{P_2}, B_2, \mathbb Z)$$ on $S_{P_2}$.
\end{corollary}

\begin{proof}
  A straightforward combination of the preceding computations.
\end{proof}

\begin{remark}
  One can normalize the $B$-fields such that at least one of $B_2 \cdot h = 1, k_4 =
    0$ and $k_s = 1$ holds, thus simplifying the above computations.
\end{remark}